\documentclass[english]{article}
\usepackage[T1]{fontenc}
\usepackage[latin9]{inputenc}
\usepackage{geometry}
\usepackage{color}
\usepackage{float}
\usepackage{amsmath}
\usepackage{amssymb}
\usepackage{graphicx}
\usepackage{esint}

\makeatletter

\providecommand{\tabularnewline}{\\}

\usepackage{epic}\usepackage{eepic}

\usepackage{color}

\usepackage{amstext}\usepackage{amsthm}\usepackage{subfigure}\usepackage{fullpage}\usepackage{color}\usepackage{multirow}\usepackage{tabulary}\usepackage{booktabs}\usepackage{enumerate}

\newtheorem{theorem}{Theorem}[section]
\newtheorem{lemma}[theorem]{Lemma}

\hfuzz=\maxdimen
\makeatother

\title{Mixed GMsFEM for the simulation of waves in highly heterogeneous media}

\author{Eric T. Chung\thanks{Department of Mathematics, The Chinese University of Hong Kong, Hong Kong SAR.
This research is partially supported by the Hong Kong RGC General Research Fund (Project number: 400813).} \;
and \; Wing Tat Leung\thanks{Department of Mathematics, Texas A\&M University, College Station, TX.}
}

\usepackage{babel}
\begin{document}

\maketitle

\begin{abstract}
Numerical simulations of waves in highly heterogenous media have important applications,
but direct computations are prohibitively expensive.
In this paper, we develop a new generalized multiscale finite element method
with the aim of simulating waves at a much lower cost.
Our method is based on a mixed Galerkin type method
with carefully designed basis functions that can capture various scales in the solution.
The basis functions are constructed based on some local snapshot spaces and local spectral problems defined on them.
The spectral problems give a natural ordering of the basis functions in the snapshot space
and allow systematically enrichment of basis functions.
In addition, by using a staggered coarse mesh,
our method is energy conserving and has block diagonal mass matrix,
which are desirable properties for wave propagation.
We will prove that our method has spectral convergence,
and present numerical results to show the performance of the method.
\end{abstract}

\section{Introduction}

Numerical simulations of waves are important in many practical areas.
For example, in computational seismology, accurate simulations of acoustic waves play
a crucial role in determining subsurface properties \cite{virieux:889,virieux:1933,saenger:SM293,GJI:GJI3620,symes:2602,masson:N33}.
Traditionally, the wave equation can be numerically solved by
finite difference methods, finite element
methods, discontinuous Galerkin methods and spectral methods \cite{basabe:562,kaser:76,Pelties:2010kx,Hermann:2011fk,GJI:GJI5221,springerlink:10.1007/s00450-010-0109-1,GJI:GJI1653,GJI:GJI967,GJI:GJI4985,CMS,JCAM-lod,nmtma,ChungEngquist06,ChungEngquist09}.
In many applications, the media of interest are highly heterogeneous
and contain many scales.
The above methods require the use of very fine meshes to
fully resolve the multiscale structure of the media.
Though the numerical solutions to the wave equation have been shown
to be accurate when the computational grid is fine enough \cite{delprat-jannaud:T37},
the practical limitations in discretization caused by limitations
in computational power restrict this accuracy.
It is therefore necessary to develop numerical approaches
that can incorporate fine-scale features into coarse-grid based methods,
where the coarse grid size is independent of
the medium scales.

There are in literature some model reduction techniques
aiming at solving the wave equation in media with multiple scales.
For instances, in \cite{AADA,Geophysics,gao2015numerical,owhadi2008numerical,korostyshevskaya2006matrix,vdovina2005operator}, some numerical homogeneization and upscaling based
techniques are developed.
In these methods, the heterogeneous medium is replaced by an effective medium
which can be efficiently resolved by a coarse mesh, giving certain reduction in computational cost.
In addition, various
multiscale methods are developed in \cite{chung2014generalized,gao2015generalized,abdulle2011finite,engquist2009multi,fish2004space,abdulle2014finite,
engquist2012multiscale,jiang2010analysis},
which also aim at discretizing the wave equation in a coarse grid by the use of the multiscale finite element method \cite{efendiev2009multiscale}
or by the use of the heterogeneous multiscale method \cite{weinan2007heterogeneous}.
While the above are very successful methods,
they can produce solutions with limited accuracy
and sometimes fail to give correct solutions.
Thus, there is a need to systematically enhance the accuracy
within the model reduction framework.
In this paper, we will focus on the recently developed Generalized Multiscale Finite Element Method (GMsFEM) \cite{egh12}.
The GMsFEM is a generalization of the classical multiscale finite element method (MsFEM) \cite{efendiev2009multiscale}
in the sense that multiple basis functions can be systematically added
to each coarse element.
The method consists of two stages: the offline stage and the online stage.
In the offline stage, a space of locally supported snapshot functions is constructed.
The snapshot space contains a large set of functions and can be used to capture essentially all fine-scale
features of the solution. A space reduction is then performed by the use of a local spectral decomposition,
and the dominant modes are taken as the multiscale basis functions.
We notice that all these computations are done before the actual simulations of the solution.
In the online stage, when a given source term or boundary condition is given, the above offline basis functions
are used to obtain an approximate solution.
One can also adaptively select the basis functions in various coarse elements in order to
obtain better efficiency and accuracy \cite{chung2014adaptive,chung2015residual}.

The purpose of this paper is to develop a new GMsFEM for the wave equation.
There are previous works on GMsFEM for the wave equation
based on the second order formulation and a discontinuous Galerkin framework \cite{chung2014generalized,gao2015generalized}.
These methods give accurate simulations of waves in coarse meshes, but are lack of energy conservation.
To develop a scheme with energy conservation, we consider the wave equation in the pressure-velocity formulation (called the mixed formulation).
We will use some ideas from our earlier work on mixed GMsFEM for high contrast flows \cite{chung2015mixed}.
However, the method in \cite{chung2015mixed} cannot be used directly for the wave equation
since it is based on a piecewise constant approximation for pressure, which is not accurate for the wave equation,
and the velocity basis functions
give mass matrix that is not block diagonal.
To derive a new GMsFEM with block diagonal mass matrix
and energy conservation, we will use a staggered mesh \cite{ChungEngquist06,ChungEngquist09,OLS,SDG-cd,SDG-curl,JCAM-meta},
where it is shown that such idea can give a numerical scheme
with block diagonal mass matrix and energy conservation.
In addition, this idea can give a smaller dispersion error \cite{dispersion,JCP-max}.
In \cite{AADA,Geophysics}, we have applied a staggered coarse mesh
in the numerical upscaling framework for the wave equation,
where only one multiscale basis function is used for each coarse region.
In this paper, we will develop a mixed GMsFEM based on a staggered mesh,
giving block diagonal mass matrix, energy conservation and systematic enrichment of multiscale basis functions.
In our new method, we will construct multiscale basis functions for both the pressure and the velocity.
The construction follows the general methodology of GMsFEM by solving local spectral problems,
which are carefully designed to achieve our goals.
The spectral problems give a natural ordering of the basis functions according to the magnitudes
of the corresponding eigenvalues.
Dominant modes, those eigenfunctions with small eigenvalues, are taken as the basis functions
and can be added systematically due to the natural ordering.
Furthermore, we prove that the method is convergent, and the convergence rate is inversely
proportional to the smallest eigenvalue of those eigenfunctions not selected in the basis set.
Another advantage of using the mixed formulation is that perfectly matched layers
can be used in conjunction with our method easily,
and we will illustrate this in our numerical experiments.

The paper is organized as follows. In Section \ref{sec:problem},
we will state our problem and define some notations as well as the staggered mesh.
In Section \ref{sec:basis} and Section \ref{sec:method}, we will give the constructions
of our multiscale basis functions and our mixed GMsFEM.
Section \ref{sec:proof} is devoted to the convergence analysis of our method.
In Section \ref{sec:num}, we will present some numerical results
to show the performance of our method.
The paper ends with a Conclusion.

\section{Problem description}
\label{sec:problem}

Let $\Omega \subset \mathbb{R}^2$ be the computational domain.
We consider the following wave equation in mixed formulation
\begin{align}
\kappa \frac{\partial v}{\partial t} +\nabla p & =0\; \quad\text{in} \;\Omega\label{eq:mixed}\\
\rho \frac{\partial p}{\partial t} +\nabla\cdot v & =f\; \quad\text{in} \;\Omega\label{eq:mixed2}
\end{align}
with homogeneous Dirichlet boundary condition $p=0$ on $\partial\Omega$.
We assume that the bulk modulus $\kappa^{-1}$ and the density $\rho$ are highly oscillatory.
The aim of the paper is to construct multiscale basis functions,
which provide accurate and efficient approximations of the pressure $p$ and the velocity $v$ on a coarse grid.
The homogeneous Dirichlet boundary condition
is chosen to simplify the discussions.
Our method can be easily applied to other types of boundary conditions
as well as perfectly matched layers.
Moreover, the extension of our method to the three dimensional case is straightforward.

First, we present the triangulation of the domain $\Omega$.
Let $\mathcal{T}^H_0$ be an initial coarse triangulation of the domain.
For each triangle in $\mathcal{T}^H_0$, we refine it into three triangles by connecting the centroid
to the three vertices.
The resulting refinement is called $\mathcal{T}^H$, which is our coarse mesh.
Our multiscale basis functions are defined in this coarse mesh $\mathcal{T}^H$.
We use $\mathcal{E}_p$ to denote the set of all edges in the initial triangulation
and use $\mathcal{E}_p^0 = \mathcal{E}_p \backslash \partial\Omega$.
Moreover, we use $\mathcal{E}_v$ to denote the set of new edges formed by the above division process.
Note that, the set of all edges $\mathcal{E}^H$ of the coarse mesh is $\mathcal{E}^H = \mathcal{E}_p \cup \mathcal{E}_v$,
and the set of all interior edges $\mathcal{E}^{H,0}$ of the coarse mesh is $\mathcal{E}^{H,0} = \mathcal{E}^0_p \cup \mathcal{E}_v$.
The fine mesh $\mathcal{T}^h$ is obtained by refining $\mathcal{T}^H$ in the conforming way.
We use $\mathcal{E}^h$ to denote the set of edges in the fine mesh $\mathcal{T}^h$.
An illustration of the above definitions is shown in Figure~\ref{fig:mesh}.
In particular, the solid lines are the edges of the initial triangulation $\mathcal{T}^H_0$.
By our construction, these lines also represent edges in $\mathcal{E}_p$.
Moreover, the dash lines represent the new edges formed by the subdivision process.
Note that, the coarse mesh $\mathcal{T}^H$ is defined as the union of all the new triangles obtained from the subdivision process.
We will discuss later the motivations and advantages of using such a mesh.

\begin{figure}[ht]
\centering
\includegraphics[scale=0.4]{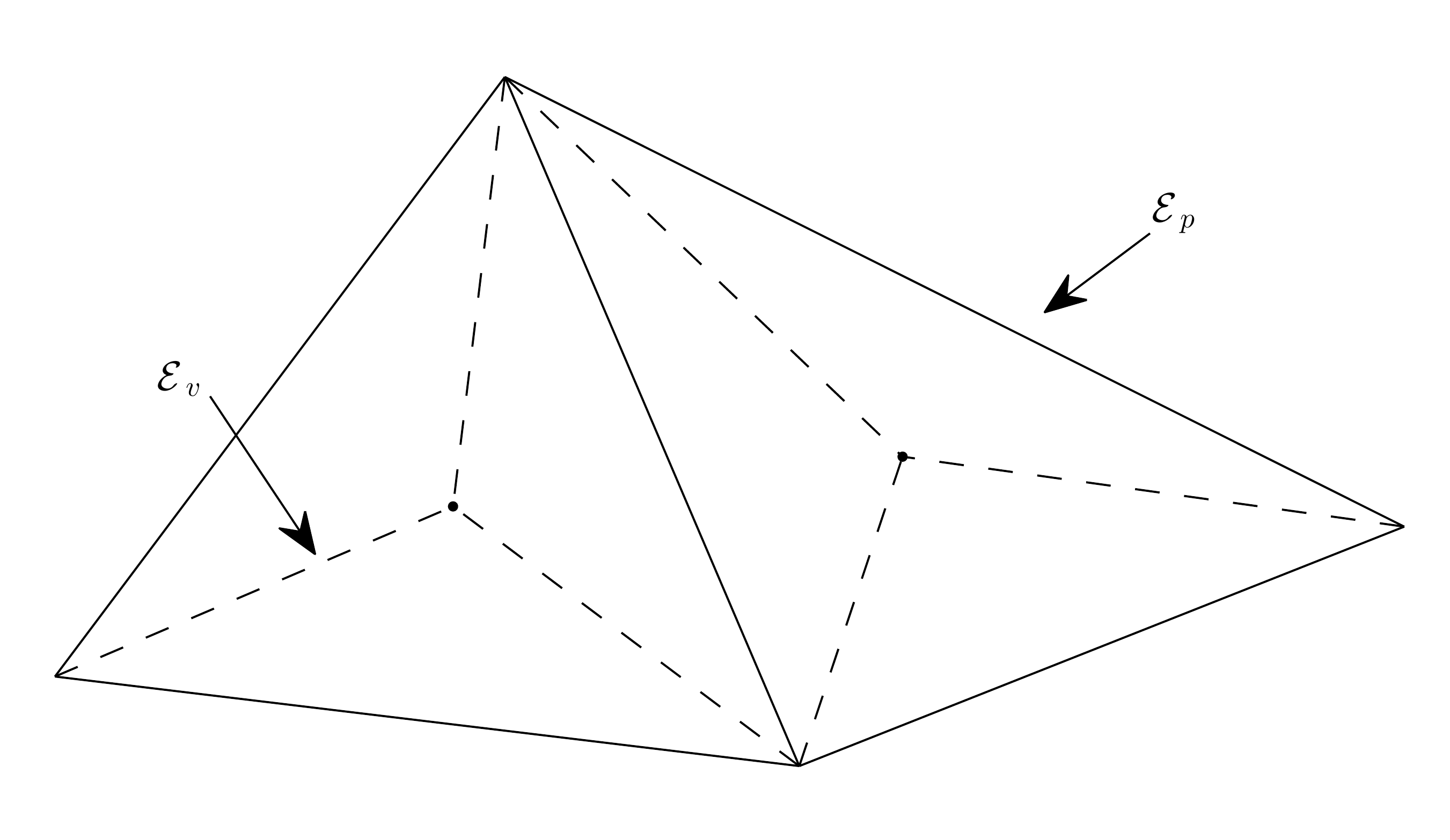}

\protect\caption{An illustration of the subdivision process and
the definition of $\mathcal{T}^H$.}
\label{fig:mesh}
\end{figure}

The wave equations (\ref{eq:mixed})-(\ref{eq:mixed2}) can be discretized
on the fine mesh $\mathcal{T}^h$ by the standard Raviart-Thomas finite element (RT0) method.
Let $(V_h,Q_h)$ be the standard RT0 space for $(v,p)$ with respect to the fine mesh $\mathcal{T}^h$.
Then, the RT0 method reads: find
$v_h \in V_h$ and $p_h \in Q_h$ such that
\begin{eqnarray}
\int_{\Omega} \kappa \frac{\partial v_h}{\partial t} \cdot w - \int_{\Omega} p_h \nabla \cdot w
 &=& 0, \quad \forall w\in V_h, \label{eq:me1} \\
\int_{\Omega} \rho \frac{\partial p_h}{\partial t} q + \int_{\Omega} q \nabla \cdot v_h
&=& \int_{\Omega} f q, \quad \forall q\in Q_h. \label{eq:me2}
\end{eqnarray}
Notice that the RT0 scheme (\ref{eq:me1})-(\ref{eq:me2}) does not have a block diagonal matrix.
We will present a modified scheme for (\ref{eq:mixed})-(\ref{eq:mixed2}) based on the above RT0 method (\ref{eq:me1})-(\ref{eq:me2}).
The resulting method has the advantage that the mass matrix is
block diagonal.
We will use similar ideas as in \cite{ChungEngquist06,ChungEngquist09}.

The main idea is to decouple the degrees of freedom for the velocity
on the subset $\mathcal{E}^h  \cap \mathcal{E}^0_p$ of fine grid edges,
which are the fine grid edges lying in $\mathcal{E}^0_p$.
In particular, we do not enforce the continuity of the normal components of velocity
on the fine grid edges in $\mathcal{E}^h  \cap \mathcal{E}^0_p$.
Moreover, we will introduce additional pressure variables
in order to penalize the normal jumps of velocity on these edges.

\begin{figure}[ht]
\centering
\includegraphics[scale=0.4]{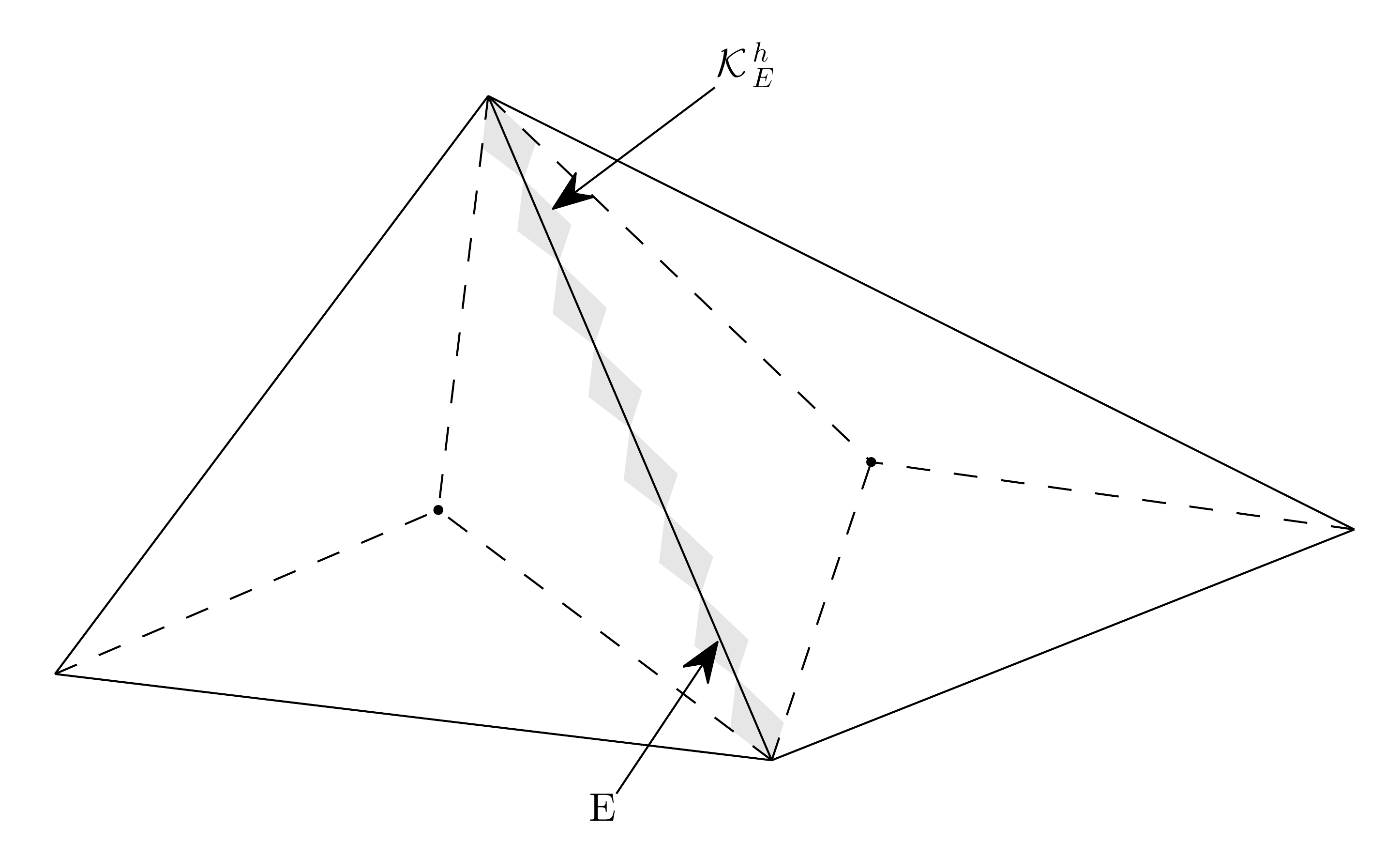}

\protect\caption{The definition of $\mathcal{K}^h_E$. The elements of $\mathcal{K}^h_E$
are shown as shaded triangles.}
\label{fig:kh}
\end{figure}

We let $\widehat{V}_h$ be the decoupled velocity space.
That is, $\widehat{V}_h$ is defined by using $V_h$ and by decoupling the normal components of velocity on
the fine grid edges in $\mathcal{E}^h  \cap \mathcal{E}^0_p$.
We will introduce an additional pressure space $\widetilde{Q}_h \subset L^2(\Omega)$ as follows.
Let $\mathcal{K}^h \subset \mathcal{T}^h$ be the set of fine mesh elements having non-empty intersection with $\mathcal{E}_p$.
We also let $\mathcal{K}^h_E \subset \mathcal{K}^h$ be a subset containing fine mesh elements
with non-empty intersection with $E\in \mathcal{E}_p$.
See Figure~\ref{fig:kh} for an illustration.
Then we define $q \in \widetilde{Q}_h$
by the following conditions:
\begin{itemize}
\item $q|_{\tau} \in P^1(\tau)$ and $q|_e \in P^0(e)$, where $\tau\in \mathcal{K}^h$ and $e$ is the edge of $\tau$ lying in $\mathcal{E}_p$;
\item $\int_{\tau} q = 0$ for all $\tau \in \mathcal{K}^h$;
\item $q$ is continuous on all edges $e \in \mathcal{E}^h \cap \mathcal{E}_p^0$.
\end{itemize}
Note that functions in $\widetilde{Q}_h$ have supports in $\mathcal{K}^h$.
Then we define $\widehat{Q}_h = Q_h + \widetilde{Q}_h$.
Note that we can impose the homogeneous Dirichlet boundary condition $p=0$ in the space $\widehat{Q}_h$
and we denote the resulting space by $\widehat{Q}_{h,0}$.
To construct $\widehat{Q}_{h,0}$,
we first define a subspace $\widetilde{Q}_{h,0} \subset \widetilde{Q}_h$
by $\widetilde{Q}_{h,0} = \{ q\in \widetilde{Q}_h \, : \,q|_e = 0, \, \forall e\in \mathcal{E}_p \cap\partial\Omega \}$.
Then we define $\widehat{Q}_{h,0} = Q_h + \widetilde{Q}_{h,0}$.
We remark that the increase in the dimension by decoupling $V_h$ to form $\widehat{V}_h$
is the same as the increase in the dimension by enriching the space $Q_h$ to form $\widehat{Q}_h$.

The modified
numerical scheme is stated as follows.
We find $v_h \in \widehat{V}_h$ and $p_h \in \widehat{Q}_{h,0}$ such that
\begin{eqnarray}
\int_{\Omega} \kappa \frac{\partial v_h}{\partial t} \cdot w - \int_{\Omega} p^{(I)}_h \nabla \cdot w
+ \sum_{e\in\mathcal{E}^h\cap \mathcal{E}_p^0} \int_e p^{(B)}_h [ w \cdot n] &=& 0, \quad \forall w\in \widehat{V}_h, \label{eq:rt1} \\
\int_{\Omega} \rho \frac{\partial p_h}{\partial t} q + \int_{\Omega} q^{(I)} \nabla \cdot v_h - \sum_{e\in\mathcal{E}^h\cap \mathcal{E}_p^0} \int_e q^{(B)} [ v_h \cdot n]
&=& \int_{\Omega} f q, \quad \forall q\in \widehat{Q}_{h,0}, \label{eq:rt2}
\end{eqnarray}
where $q^{(I)}$ and $q^{(B)}$ denote the components of $q$ in the spaces $Q_h$ and $\widetilde{Q}_h$
respectively, for any $q\in\widehat{Q}_h$.
The solution $(v_h,p_h)$ of (\ref{eq:rt1})-(\ref{eq:rt2}) is considered as the reference solution.
In the following sections, we will construct multiscale solution $(v_H,p_H)$ that gives good approximation of $(v_h,p_h)$
and derive the corresponding error bound.
Since the purpose of this paper is the construction and the analysis of a new multiscale method,
the error analysis for the scheme (\ref{eq:rt1})-(\ref{eq:rt2}) is not considered in this paper.

\section{Multiscale basis functions}
\label{sec:basis}

In this section,
we will introduce multiscale basis functions
and give the constructions of the multiscale approximation spaces $V_H$ and $Q_H$
for approximating $v_h$ and $p_h$ respectively.
We emphasize that the basis functions and the corresponding multiscale method are defined with respect to the coarse mesh $\mathcal{T}^H$.

For a coarse element $K\in\mathcal{T}^H$,
we define $V_h(K)$ as the restriction of $V_h$ in $K$
and $Q_h(K)$ as the restriction of $Q_h$ in $K$.
In addition, $V_{h,0}(K)$ is the subspace of $V_h(K)$
and contains vector fields whose normal components are zero on $\partial K$.

Below, we will give the definitions for the spaces $Q_H$ and $V_H$.
Our generalized multiscale finite element method reads:
find $(v_H,p_H) \in V_H \times Q_H$ such that
\begin{eqnarray}
\int_{\Omega} \kappa \frac{\partial v_H}{\partial t} \cdot w - \int_{\Omega} p_H \nabla \cdot w
 &=& 0, \quad \forall w\in V_H, \label{eq:main1} \\
\int_{\Omega} \rho \frac{\partial p_H}{\partial t} q + \int_{\Omega} q \nabla \cdot v_H
&=& \int_{\Omega} f q, \quad \forall q\in Q_H. \label{eq:main2}
\end{eqnarray}

There are totally three set of basis functions.
The first set of basis functions can be considered as a generalization of the RT0 element.
The second set of basis functions gives enrichments of the normal component of velocity
across coarse grid edges.
The third set of basis functions corresponds to the standing modes within coarse elements
with zero boundary conditions.

\subsection{The first basis set}

We use $V_H^{(1)}$ and $Q_H^{(1)}$ to denote the first set of basis functions for the velocity $v$ and the pressure $p$.
The space $Q_H^{(1)}$ is taken as the piecewise constant space with respect to the coarse mesh $\mathcal{T}^H$.
We will define the space $V_H^{(1)}$ as follows.
For each coarse edge $E \in \mathcal{E}^H$, we define $\omega_E$
as the union of all coarse elements having the edge $E$.
For each $E \in \mathcal{E}^H$, we define one basis function $\phi_E^{(1)}$ whose support is $\omega_E$.
The basis $\phi_E^{(1)}$ is defined by finding $(\phi_E^{(1)}, p_E^{(1)}) \in V_h(K) \times Q_h(K)$ such that
\begin{eqnarray}
\int_{K} \kappa \, \phi_E^{(1)} \cdot w - \int_{K} p_E^{(1)} \nabla \cdot w  &=& 0, \quad \forall w\in V_{h,0}(K), \label{eq:local1} \\
 \int_{K} q \nabla \cdot \phi_E^{(1)}
&=& \int_{K} c_E q, \quad \forall q\in Q_h(K) \label{eq:local2}
\end{eqnarray}
for each $K \subset \omega_E$, with the following boundary condition
\begin{equation}
\phi_E^{(1)} \cdot n =
\begin{cases}
1, \quad &\text{on }\,  E, \\
0, \quad &\text{on } \,\partial K \backslash E
\end{cases}
\end{equation}
where $n$ denote the unit normal vectors on edges.
In (\ref{eq:local2}), the constant $c_E = |E| / |K|$.
The space $V_H^{(1)}$ is defined by
\begin{equation}
V_H^{(1)} = \text{span} \Big\{ \phi_E^{(1)} \, : \, E \in \mathcal{E}^H \Big\}.
\end{equation}
We remark that scheme (\ref{eq:main1})-(\ref{eq:main2}) with $V_H = V_H^{(1)}$ and $Q_H = Q_H^{(1)}$
can be regarded as a generalization of the classical RT0 method.

\subsection{The second basis set}

We use $V_H^{(2)}$ to denote the second set of basis functions for the velocity $v$.
We will see below that this corresponds to enrichments of velocity with respect to coarse grid edges.
Notice that, in the second basis set, we only enrich the approximation space for the velocity.
Let $E\in\mathcal{E}^H$.
Note that $E\cap \mathcal{E}^h$ defines a partition for $E$ by fine grid edges.
We define $D_E$ be the space of piecewise constant functions
on $E$ with respect to the partition $E\cap \mathcal{E}^h$.
Let $\overline{D}_E$ be the subspace of $D_E$ containing functions with zero mean.
We write $\overline{D}_E = \text{span} \{ \delta_{E,i} \, : \,  i=1,2,\cdots, N_E \}$, where $N_E$ is the dimension of $\overline{D}_E$.
For each $E \in \mathcal{E}^H$, we define a set of basis functions $\phi_{E,j}^{(2)}$ whose support is $\omega_E$.
We first find $(\psi_{E,i}^{(2)}, q_{E,i}^{(2)}) \in V_h(K) \times Q_h(K)$ such that
\begin{eqnarray}
\int_{K} \kappa \, \psi_{E,i}^{(2)} \cdot w - \int_{K} q_{E,i}^{(2)} \nabla \cdot w  &=& 0, \quad \forall w\in V_{h,0}(K), \label{eq:local3} \\
 \int_{K} q \nabla \cdot \psi_{E,i}^{(2)}
&=& 0, \quad \forall q\in Q_h(K) \label{eq:local4}
\end{eqnarray}
for each $K \subset \omega_E$, with the following boundary condition
\begin{equation}
\psi_{E,i}^{(2)} \cdot n =
\begin{cases}
\delta_{E,i}, \quad &\text{on }\,  E, \\
0, \quad &\text{on } \,\partial K \backslash E.
\end{cases}
\end{equation}
We next define a snapshot space $V_{E,\text{snap}}^{(2)}$ on $E$ by
\begin{equation}
V_{E,\text{snap}}^{(2)} = \text{span} \Big\{ \psi_{E,i}^{(2)} \, : \,  \forall \, \delta_{E,i} \in \overline{D}_E \Big\}.
\end{equation}
We remark that the snapshot space $V_{E,\text{snap}}^{(2)}$ is large and dimension reduction is necessary.
To obtain a reduced dimension space,
we use the following spectral problem. Find $\phi\in V_{E,\text{snap}}^{(2)}$ and $\lambda\in\mathbb{R}$ such that
\begin{equation}
\int_E  (\phi \cdot n) \, (w\cdot n) = \lambda \int_{\omega_E} \kappa\, \phi \cdot w, \quad \forall w \in V_{E,\text{snap}}^{(2)}.
\label{eq:spectral}
\end{equation}
We arrange the eigenvalues in increasing order, that is, $\lambda_{E,1} < \lambda_{E,2} < \cdots < \lambda_{E, N_E}$.
For each coarse edge $E$, we take the first $n_E$ eigenfunctions $\phi_{E,j}^{(2)}$. We then define
\begin{equation}
V_{E,\text{off}}^{(2)} = \text{span} \Big\{ \phi_{E,j}^{(2)} \, : \,  j=1,2,\cdots, n_E \Big\}.
\end{equation}
Finally, the space $V_H^{(2)}$ is defined as
\begin{equation}
V_H^{(2)} = \sum_{E\in\mathcal{E}^H} V_{E,\text{off}}^{(2)}.
\end{equation}

\subsection{The third basis set}

We use $V_H^{(3)}$ and $Q_H^{(3)}$ to denote the third set of basis functions for the velocity $v$ and the pressure $p$.
We will see below that these basis functions correspond to some standing modes within coarse elements.
Let $K\in\mathcal{T}^H$ be a given coarse element.
We consider the following spectral problem: find $\psi \in V_{h,0}(K)$, $p\in Q_h(K)$ and $\mu\in\mathbb{R}$
such that
\begin{eqnarray}
\int_{K} \kappa \, \psi \cdot w - \int_{K} p \nabla \cdot w  &=& 0, \quad \forall w\in V_{h,0}(K), \label{eq:local5} \\
 \int_{K} q \nabla \cdot \psi
&=& \mu \int_{K} \rho \, p \, q, \quad \forall q\in Q_h(K) \label{eq:local6}
\end{eqnarray}
subject to the zero mean condition $\int_K p = 0$.
Notice that the spaces $V_{h,0}(K)$ and $Q_h(K)$ are considered as the snapshot spaces.
We arrange the eigenvalues in increasing order, that is, $\mu_{K,1} < \mu_{K,2} < \cdots < \mu_{K,M_K}$,
where $M_K+1$ is the dimension of $Q_{h}(K)$.
We will take the first $m_K$ eigenfunctions $(\phi_{K,j}^{(3)}, p_{K,j}^{(3)})$ as basis.
We define
\begin{equation}
V_{K,\text{off}}^{(3)} = \text{span} \Big\{ \phi_{K,j}^{(3)} \, : \, j=1,2,\cdots, m_K \Big\}
\end{equation}
and
\begin{equation}
Q_{K,\text{off}}^{(3)} = \text{span} \Big\{ p_{K,j}^{(3)} \, : \, j=1,2,\cdots, m_K \Big\}.
\end{equation}
Then we define
\begin{equation}
V_H^{(3)} = \sum_{K\in \mathcal{T}^H} V_{K,\text{off}}^{(3)} \quad \text{ and } \quad
Q_H^{(3)} = \sum_{K\in \mathcal{T}^H} Q_{K,\text{off}}^{(3)}.
\end{equation}
Finally, we can take $V_H = V_H^{(1)} + V_H^{(2)}+ V_H^{(3)}$ and $Q_H = Q_H^{(1)} + Q_H^{(3)}$
as the approximation spaces
in the scheme (\ref{eq:main1})-(\ref{eq:main2}).

\section{The mixed GMsFEM}
\label{sec:method}

As discussed in the previous section,
we can define the mixed GMsFEM for the wave equation by the system
(\ref{eq:main1})-(\ref{eq:main2}) together with
the choices $V_H = V_H^{(1)} + V_H^{(2)}+ V_H^{(3)}$ and $Q_H = Q_H^{(1)} + Q_H^{(3)}$
as the approximation spaces.
However, the resulting scheme
does not have a block diagonal mass matrix, and is
therefore slow in the time stepping process.
The reason for getting a non-diagaonal mass matrix is that the basis functions in $V_H^{(1)}$ and $V_H^{(2)}$
have overlapping supports on two adjacent coarse elements.
In particular, the basis functions for the velocity are coupled through the coarse grid edges $\mathcal{E}^H$.
In the following,
we will modify these spaces so that the velocity basis functions have disjoint supports.
One important feature is that we will only decouple the basis functions on the coarse grid edges
belonging to $\mathcal{E}_p^0$ only.
We do not decouple the velocity basis functions for all coarse edges.
The resulting velocity basis functions have disjoint supports
on triangles of the initial coarse mesh $\mathcal{T}^H_0$, see Figure~\ref{fig:mesh}.

Let $E \in \mathcal{E}^0_p$ be a coarse edge.
We recall that the basis functions in the space $V_{H}^{(1)}$ and $V_{E,\text{off}}^{(2)}$
have supports in $\omega_E$
and their normal components have common values on $E$.
We will decouple this continuity, and call the resulting space as $\widehat{V}_{H}^{(1)}$ and $\widehat{V}_{E,\text{off}}^{(2)}$ respectively.
In particular, the normal component of velocity does not necessarily have a single value on $E$.
Globally, we define
\begin{equation}
\widehat{V}_H^{(2)} = \sum_{E\in\mathcal{E}^H} \widehat{V}_{E,\text{off}}^{(2)}.
\end{equation}

We will introduce a new pressure space in order to penalize the normal continuity of velocity on $E\in\mathcal{E}_p^0$.
Let $E\in\mathcal{E}_p$. This new pressure space is denoted by $Q_{E,\text{off}}^{(2)}$, and its dimension is $n_E+1$.
We recall that $\mathcal{K}_E^h \subset \mathcal{K}^h$ is the subset of fine elements having non-empty intersection with $E$, see Figure~\ref{fig:kh}.
The basis functions of $Q_{E,\text{off}}^{(2)}$
have supports on $\mathcal{K}_E^h$.
Each $q \in Q_{E,\text{off}}^{(2)}$ is defined by the following conditions:
\begin{itemize}
\item $q|_{\tau} \in P^1(\tau)$ and $q|_e \in P^0(e)$, where $\tau\in \mathcal{K}_E^h$, and $e$ is the edge of $\tau$ lying in $E$;
\item $q$ is continuous on $E$, and $q = \phi \cdot n$ on $E$ for some $\phi \in V_H^{(1)} \cup V_H^{(2)}$;
\item $\int_{\tau} q = 0$ for all $\tau\in \mathcal{K}_E^h$.
\end{itemize}
We write
\begin{equation*}
Q_{E,\text{off}}^{(2)} = \text{span} \Big\{ p_{E}^{(1)} \Big\} + \text{span} \Big\{ p_{E,j}^{(2)} \, : \, j=1,2,\cdots, n_E \Big\}
\end{equation*}
where $p_{E}^{(1)}=1$ on $E$ and $p_{E,j}^{(2)} = \phi_{E,j}^{(2)} \cdot n$ on $E$.
We define
\begin{equation}
\widehat{Q}_H^{(2)} = \sum_{E\in \mathcal{E}_p} Q_{E,\text{off}}^{(2)}.
\end{equation}
Next, we discuss the boundary condition.
We consider Dirichlet boundary condition $p=0$ for pressure.
We will need to modify the space $\widehat{Q}_H^{(2)}$ for the boundary condition.
We define $\widehat{Q}_{H,0}^{(2)}$ by
\begin{equation}
\widehat{Q}_{H,0}^{(2)} = \sum_{E\in \mathcal{E}_p^0} Q_{E,\text{off}}^{(2)}
\end{equation}
where the sum is taken over all interior edges $E\in\mathcal{E}_p^0$.
We can take $V_H = \widehat{V}_H^{(1)} + \widehat{V}_H^{(2)}+ V_H^{(3)}$ and $Q_H = Q_H^{(1)} + \widehat{Q}_{H,0}^{(2)} + Q_H^{(3)}$.
Since the space $V_H$ is not in $H(\text{div})$, we will replace the variational form (\ref{eq:main1})-(\ref{eq:main2})
by the following.
We find $(v_H,p_H) \in V_H \times Q_H$ such that
\begin{eqnarray}
\int_{\Omega} \kappa \frac{\partial v_H}{\partial t} \cdot w - \int_{\Omega} p^{(I)}_H \nabla \cdot w
 + \sum_{E \in \mathcal{E}_p^0} \int_E p_H^{(B)} [ w \cdot n] &=& 0, \quad \forall w\in V_H, \label{eq:main3} \\
\int_{\Omega} \rho \frac{\partial p_H}{\partial t} q + \int_{\Omega} q^{(I)} \nabla \cdot v_H  - \sum_{E\in  \mathcal{E}_p^0} \int_E q^{(B)} [ v_H \cdot n]
&=& \int_{\Omega} f q, \quad \forall q\in Q_H \label{eq:main4}
\end{eqnarray}
where we use $q^{(B)}$ to denote the component of $q$ in $\widehat{Q}_{H,0}^{(2)}$
and $q^{(I)}$ to denote the component of $q$ in $Q_H^{(1)} + Q_H^{(3)}$,
for any $q\in Q_H$.
Equations (\ref{eq:main3}) and (\ref{eq:main4}) give our mixed GMsFEM
for the wave equation (\ref{eq:mixed})-(\ref{eq:mixed2}).



For the time discretization,
we will apply the leap-frog scheme to (\ref{eq:main3})-(\ref{eq:main4}).
The velocity $v_H$ are approximated at times $t_n = n \Delta t$, and the pressure $p_H$
are approximated at times $t_{n+\frac{1}{2}} = (n+\frac{1}{2}) \Delta t$, $n=0,1,\cdots$.
We let $v_H^n$ and $p_H^{n+\frac{1}{2}}$ be the approximate solutions at times $t_n$ and $t_{n+\frac{1}{2}}$ respectively.
The leap-frog scheme reads:
\begin{eqnarray}
\int_{\Omega} \kappa\, \frac{v_H^{n+1}-v_H^n}{\Delta t} \cdot w - \int_{\Omega} p^{(I,n+\frac{1}{2})}_H \nabla \cdot w
 + \sum_{E \in \mathcal{E}_p^0} \int_E p_H^{(B,n+\frac{1}{2})} [ w \cdot n] &=& 0, \quad \forall w\in V_H, \label{eq:main5} \\
\int_{\Omega} \rho\, \frac{p_H^{n+\frac{3}{2}}  - p_H^{n+\frac{1}{2}}}{\Delta t} q + \int_{\Omega} q^{(I)} \nabla \cdot v^{n+1}_H  - \sum_{E\in  \mathcal{E}_p^0} \int_E q^{(B)} [ v^{n+1}_H \cdot n]
&=& \int_{\Omega} f(t_{n+1},\cdot) q, \quad \forall q\in Q_H \label{eq:main6}
\end{eqnarray}
where $p_H^{(B,n+\frac{1}{2})}$ denotes the component of $p_H^{n+\frac{1}{2}}$ in $\widehat{Q}_{H,0}^{(2)}$
and $p_H^{(I,n+\frac{1}{2})}$ denotes the component of $p_H^{n+\frac{1}{2}}$ in $Q_H^{(1)} + Q_H^{(3)}$.
We remark that the stability estimate for $\Delta t$
can be obtained by standard techniques and the inverse estimate, see for example \cite{Geophysics}.

\section{Theory}
\label{sec:proof}

In this section, we will prove the stability and convergence of the
mixed GMsFEM using the basis constructed in Section \ref{sec:method}.
We will also show that our method is energy conserving.
To do these,
we will first introduce some
preliminary definitions which will be used in our analysis. Then,
we will prove the energy conservation,
stability and convergence for the mixed GMsFEM (\ref{eq:main3})-(\ref{eq:main4})
with the use of
$V_H = \widehat{V}_H^{(1)} + \widehat{V}_H^{(2)}+ V_H^{(3)}$ and $Q_H = Q_H^{(1)} + \widehat{Q}_{H,0}^{(2)} + Q_H^{(3)}$
as the approximation spaces.
We will assume zero initial values to simplify notations and analysis.
We will also assume that the wave equation is solved from the initial time $t=0$
to some finite time $t=T>0$.

We define the inner product for the space $\widehat{V}_h$ by
\begin{align*}
(v,w)_{V} & =\int_{\Omega}\kappa\, v\cdot w
\end{align*}
and the inner product for the space $\widehat{Q}_{h,0}$ by
\[
(p,q)_{Q}=\int_{\Omega} \rho \, p \, q.
\]
Then, the corresponding norms in the spaces $\widehat{V}_h$ and $\widehat{Q}_{h,0}$
are induced by the inner products
and are defined as
\begin{equation}
\|v\|_{V}^{2}=(v,v)_{V} \quad \text{ and } \quad \|p\|_{Q}^{2}=(p,p)_{Q}.
\label{eq:norm}
\end{equation}
Note that the above inner products and norms
are also well-defined for $V_H$ and $Q_H$
since $V_H \subset \widehat{V}_h$ and $Q_H \subset \widehat{Q}_{h,0}$.

\subsection{Energy conservation and stability}

In this section,
we will prove that the method (\ref{eq:main3})-(\ref{eq:main4}) is energy conserving
and is stable with respect to the norms (\ref{eq:norm}).
In particular, we will prove the following theorem.

\begin{theorem}
Let $v_H \in V_H$ and $p_H\in Q_H$ be the solutions of (\ref{eq:main3})-(\ref{eq:main4}). Then we have the following energy conservation property
\begin{equation}
\frac{d}{dt} \frac{1}{2} \big( \|v_H\|_V^2 + \|p_H\|_Q^2 \big) = 0
\label{eq:conservation}
\end{equation}
provided $f=0$. Moreover, the following stability holds
\begin{equation}
\max_{0\leq t\leq T} \Big( \| v_H(t, \cdot )\|^2_V + \|p_H(t,\cdot)\|^2_Q \Big)
 \leq  4 \Big( \int_0^T \| \rho^{-1} f \|_Q \Big)^2.
 \label{eq:stab}
\end{equation}
\end{theorem}
\begin{proof}
We take $w = v_H$ and $q = p_H$ in (\ref{eq:main3})-(\ref{eq:main4}),
\begin{equation}
\int_{\Omega} \kappa \frac{\partial v_H}{\partial t} \cdot v_H + \int_{\Omega} \rho \frac{\partial p_H}{\partial t} p_H
= \int_{\Omega} f p_H.
\label{eq:energy}
\end{equation}
Thus, (\ref{eq:conservation}) holds when $f=0$.

In addition, using (\ref{eq:energy}) and the Cauchy-Schwarz inequality, we have
\begin{equation}
\frac{d}{dt} \frac{1}{2} \big( \|v_H\|_V^2 + \|p_H\|_Q^2 \big) \leq \| \rho^{-1} f \|_Q \, \|p_H\|_Q.
\end{equation}
Integrating in time from $0$ to $s$,
\begin{equation*}
\| v_H(s, \cdot )\|^2_V + \|p_H(s,\cdot)\|^2_Q \leq 2 \int_0^T \| \rho^{-1} f \|_Q \, \|p_H(t,\cdot)\|_Q \leq 2 \max_{0\leq t\leq T} \|p_H\|_Q \int_0^T \| \rho^{-1} f \|_Q.
\end{equation*}
So, we have
\begin{equation*}
\max_{0\leq t\leq T} \Big( \| v_H(t, \cdot )\|^2_V + \|p_H(t,\cdot)\|^2_Q \Big)
 \leq 2 \max_{0\leq t\leq T} \|p_H(t,\cdot)\|_Q \int_0^T \| \rho^{-1} f \|_Q.
\end{equation*}
Using the Cauchy-Schwarz inequality,
\begin{equation*}
\max_{0\leq t\leq T} \Big( \| v_H(t, \cdot )\|^2_V + \|p_H(t,\cdot)\|^2_Q \Big)
 \leq \frac{1}{2} \max_{0\leq t\leq T} \|p_H(t,\cdot)\|^2_Q + 2 \Big( \int_0^T \| \rho^{-1} f \|_Q \Big)^2.
\end{equation*}
This implies (\ref{eq:stab}).

\end{proof}

\subsection{Convergence}

In this section, we will prove the convergence of the mixed GMsFEM (\ref{eq:main3})-(\ref{eq:main4}).
Note that we can write $p_h^{(I)} \in Q_h$ in the following way
\begin{equation}
p_h^{(I)} = \sum_{K\in\mathcal{T}^H} \Big( a_{K,0} + \sum_{j=1}^{M_K} a_{K,j} p_{K,j}^{(3)} \Big)
\label{eq:f1}
\end{equation}
where $a_{K,0}$ is the average value of $p_h^{(I)}$ on $K$, and $a_{K,j}$ are determined by
\begin{equation*}
a_{K,j} = (p_h^{(I)}, p_{K,j}^{(3)})_Q, \quad\quad j=1,2,\cdots, M_K.
\end{equation*}
We remark that the dimension of $Q_h(K)$ is $M_K+1$.
In addition, we can write $p_h^{(B)} \in \widetilde{Q}_{h,0}$ in the following way
\begin{equation}
p_h^{(B)} = \sum_{E\in\mathcal{E}_p^0} \Big( b_{E,0}  p_{E}^{(1)}+ \sum_{j=1}^{N_E} b_{E,j} p_{E,j}^{(2)} \Big)
\label{eq:f2}
\end{equation}
where $b_{E,0}$ is the average value of $p_h^{(B)}$ on $E$, and $b_{E,j}$ are determined by
\begin{equation*}
b_{E,j} = \int_E p_h^{(B)} \phi_{E,j}^{(2)}\cdot n, \quad\quad j=1,2,\cdots, N_E.
\end{equation*}
On the other hand, we can write $v_h \in \widehat{V}_h$ in the following way
\begin{equation}
v_h = \sum_{E\in\mathcal{E}^H} \Big( c_{E,0} \phi_E^{(1)} + \sum_{j=1}^{N_E} c_{E,j} \phi_{E,j}^{(2)} \Big)
+ \sum_{K\in\mathcal{T}^H} \sum_{j=1}^{M_K} d_{K,j} \phi_{K,j}^{(3)}
\label{eq:f3}
\end{equation}
where $c_{E,0}$ is the average value of $v_h \cdot n$ on $E$, $c_{E,j}$ and $d_{K,j}$ are determined by
\begin{equation*}
c_{E,j} = \int_E (v_h \cdot n) (\phi_{E,j}^{(2)} \cdot n), \quad\quad j=1,2,\cdots, N_E
\end{equation*}
and
\begin{equation*}
d_{E,j} = \int_{K} \kappa\, v_h \cdot \phi_{K,j}^{(3)}, \quad\quad j=1,2,\cdots, M_K
\end{equation*}
respectively.

Using (\ref{eq:f1})-(\ref{eq:f3}),
we define the following interpolants $\pi(p_h^{(I)}) \in Q_H, \pi(p_h^{(B)}) \in Q_H$
and $\pi(v_h) \in V_H$ by
\begin{equation}
\begin{split}
\pi(p_h^{(I)}) &= \sum_{K\in\mathcal{T}^H} \Big( a_{K,0} + \sum_{j=1}^{m_K} a_{K,j} p_{K,j}^{(3)} \Big), \\
\pi(p_h^{(B)}) &= \sum_{E\in\mathcal{E}_p^0} \Big( b_{E,0} + \sum_{j=1}^{n_E} b_{E,j} p_{E,j}^{(2)} \Big), \\
\pi(v_h) &= \sum_{E\in\mathcal{E}^H} \Big( c_{E,0} \phi_E^{(1)} + \sum_{j=1}^{n_E} c_{E,j} \phi_{E,j}^{(2)} \Big)
+ \sum_{K\in\mathcal{T}^H} \sum_{j=1}^{m_K} d_{K,j} \phi_{K,j}^{(3)}.
\end{split}
\label{eq:inter}
\end{equation}
Notice that in the above definitions of interpolants,
 we only sum over all basis functions in the spaces $Q_H$ and $V_H$,
 instead of suming over all functions as in (\ref{eq:f1})-(\ref{eq:f3}).

Next, we prove the following lemma concerning some properties of the interpolants defined in (\ref{eq:inter}).
\begin{lemma}
\label{lem:inter}
The interpolants $\pi(p_h^{(I)}) \in Q_H, \pi(p_h^{(B)}) \in Q_H$
and $\pi(v_h) \in V_H$ defined in (\ref{eq:inter}) satisfy the following conditions
\begin{eqnarray}
\int_{\Omega} (p_h^{(I)}-\pi(p^{(I)}_h)) \nabla \cdot w  &=& 0 \label{eq:inter1} \\
 \sum_{E \in \mathcal{E}_p^0} \int_E (p_h^{(B)}-\pi(p_h^{(B)})) [ w \cdot n] &=& 0  \label{eq:inter2}
 \end{eqnarray}
 for all $w\in V_H$, and
 \begin{eqnarray}
 \int_{\Omega} q^{(I)} \nabla \cdot (v_h-\pi(v_h)) &=& 0 \label{eq:inter3} \\
 \sum_{E\in  \mathcal{E}_p^0} \int_E q^{(B)} [ (v_h-\pi(v_h)) \cdot n] &=& 0 \label{eq:inter4}
\end{eqnarray}
for all $q\in Q_H$.
\end{lemma}
\begin{proof}
By definition, we have
\begin{equation*}
p_h^{(I)}-\pi(p^{(I)}_h) = \sum_{K\in\mathcal{T}^H} \Big( \sum_{j=m_K+1}^{M_K} a_{K,j} p_{K,j}^{(3)} \Big).
\end{equation*}
Taking $q= p_{K,j}^{(3)}$ in (\ref{eq:local2}), we have
\begin{equation*}
 \int_{K} p_{K,j}^{(3)} \nabla \cdot \phi_E^{(1)}
= \int_{K} c_E \, p_{K,j}^{(3)} = 0
\end{equation*}
since $p_{K,j}^{(3)}$ has zero mean on $K$. We see that (\ref{eq:inter1}) holds when $w\in V_H^{(1)}$.
Now we take $q= p_{K,j}^{(3)}$ in (\ref{eq:local4}) to get
\begin{equation*}
 \int_{K} p_{K,j}^{(3)} \nabla \cdot \phi_{E,i}^{(2)}
 = 0, \quad\quad \forall \, i=1,2\cdots, n_E.
\end{equation*}
This shows that (\ref{eq:inter1}) holds when $w\in V_H^{(2)}$.
Finally, using $q= p_{K,j}^{(3)}$ in (\ref{eq:local6}), we have
\begin{equation*}
 \int_{K} p_{K,j}^{(3)} \nabla \cdot \phi_{E,i}^{(3)}
 = \int_K \rho \, r \, p_{K,j}^{(3)}, \quad\quad \forall \, i=1,2\cdots, m_K
\end{equation*}
where $r = \text{span}\{ p_{K,i}^{(3)} \, : \, i=1,2,\cdots, m_K \}$.
This shows that (\ref{eq:inter1}) holds when $w\in V_H^{(3)}$.

By definition, we have
\begin{equation*}
p_h^{(B)}-\pi(p_h^{(B)}) = \sum_{E\in\mathcal{E}_p^0} \Big(  \sum_{j=n_E+1}^{N_E} b_{E,j} p_{E,j}^{(2)} \Big).
\end{equation*}
It is clear that (\ref{eq:inter2}) holds when $w\in V_H^{(3)}$
since $w$ vanishes on all coarse edges.
It is also clear that (\ref{eq:inter2}) holds when $w\in V_H^{(1)}$
since $w\cdot n$ is a constant function on all coarse edges
and $p_{E,j}^{(2)}$ has zero average on all coarse edges.
By the spectral problem (\ref{eq:spectral}), $p_{E,j}^{(2)}$ is orthogonal to $w\cdot n$ on $E\in \mathcal{E}_p^0$ for all $w\in V_H^{(2)}$, so
equation (\ref{eq:inter2}) holds when $w\in V_H^{(2)}$.

By definition, we have
\begin{equation*}
v_h - \pi(v_h) = \sum_{E\in\mathcal{E}^H} \Big(  \sum_{j=n_E+1}^{N_E} c_{E,j} \phi_{E,j}^{(2)} \Big)
+ \sum_{K\in\mathcal{T}^H} \sum_{j=m_K+1}^{M_K} d_{K,j} \phi_{K,j}^{(3)}.
\end{equation*}
The proofs for (\ref{eq:inter3}) and (\ref{eq:inter4}) are similar to the proofs of (\ref{eq:inter1}) and (\ref{eq:inter2})
and can be obtained by the same techniques above.
\end{proof}

Now we prove the convergence.
Substracting (\ref{eq:main3})-(\ref{eq:main4}) from (\ref{eq:rt1})-(\ref{eq:rt2}),
\begin{eqnarray}
\int_{\Omega} \kappa \frac{\partial (v_h-v_H)}{\partial t} \cdot w - \int_{\Omega} (p_h^{(I)}-p^{(I)}_H) \nabla \cdot w
 + \sum_{E \in \mathcal{E}_p^0} \int_E (p_h^{(B)}-p_H^{(B)}) [ w \cdot n] &=& 0, \label{eq:error1} \\
\int_{\Omega} \rho \frac{\partial (p_h-p_H)}{\partial t} q + \int_{\Omega} q^{(I)} \nabla \cdot (v_h-v_H)  - \sum_{E\in  \mathcal{E}_p^0} \int_E q^{(B)} [ (v_h-v_H) \cdot n]
&=& 0, \label{eq:error2}
\end{eqnarray}
for all $w\in V_H$ and $q\in Q_H$.
Using Lemma \ref{lem:inter},
\begin{eqnarray}
\int_{\Omega} \kappa \frac{\partial (v_h-v_H)}{\partial t} \cdot w - \int_{\Omega} (\pi(p_h^{(I)})-p^{(I)}_H) \nabla \cdot w
 + \sum_{E \in \mathcal{E}_p^0} \int_E (\pi(p_h^{(B)})-p_H^{(B)}) [ w \cdot n] &=& 0, \label{eq:error3} \\
\int_{\Omega} \rho \frac{\partial (p_h-p_H)}{\partial t} q + \int_{\Omega} q^{(I)} \nabla \cdot (\pi(v_h)-v_H)  - \sum_{E\in  \mathcal{E}_p^0} \int_E q^{(B)} [ (\pi(v_h)-v_H) \cdot n]
&=& 0, \label{eq:error4}
\end{eqnarray}
for all $w\in V_H$ and $q\in Q_H$.
Taking $w = \pi(v_h)-v_H$ and $q = \pi(p_h^{(I)})-p^{(I)}_H + \pi(p_h^{(B)})-p_H^{(B)}$,
and adding the resulting equations, we have
\begin{equation}
\int_{\Omega} \kappa \frac{\partial (v_h-v_H)}{\partial t} \cdot (\pi(v_h)-v_H)
+ \int_{\Omega} \rho \frac{\partial (p_h-p_H)}{\partial t} (\pi(p_h^{(I)}) + \pi(p_h^{(B)})-p_H) = 0.
\end{equation}
Thus, we obtain
\begin{equation}
\begin{split}
&\: \max_{0\leq t\leq T} \Big( \| \pi(v_h)-v_H \|_V + \| \pi(p_h^{(I)}) + \pi(p_h^{(B)})-p_H\|_Q \Big) \\
\leq
&\: \int_0^T \Big( \| \pi(\dot{v}_h)-\dot{v}_h \|_V + \| \pi(\dot{p}_h^{(I)}) + \pi(\dot{p}_h^{(B)})-\dot{p}_h\|_Q \Big)
\end{split}
\label{eq:errorbound}
\end{equation}
where the dots denote time derivatives.
Hence, it suffices to estimate the interpolation errors.

In the following two theorems, we give error estimates of the interpolants.
We remark that the notation $\alpha \lesssim \beta$ means that $\alpha \leq c \beta$
for some constant $c>0$ independent of the mesh.
\begin{theorem}
\label{thm:inter1}
The interpolant $\pi(v_h)$ defined in (\ref{eq:inter}) satisfies
\begin{equation}
\| v_h - \pi(v_h)\|_V^2 \lesssim
\Big( \max_{E\in\mathcal{T}^H} \lambda_{E,n_E+1}^{-1} \Big) \sum_{E\in\mathcal{E}^H} \int_E (v_h\cdot n)^2
+
\Big( \max_{K\in\mathcal{T}^H} \mu_{K,m_K+1}^{-1} \Big) \| \rho^{-1} f - \dot{p}_h \|_Q^2.
\end{equation}
\end{theorem}
\begin{proof}
By definition,
we have
\begin{equation*}
v_h - \pi(v_h) = \sum_{E\in\mathcal{E}^H} \Big(  \sum_{j=n_E+1}^{N_E} c_{E,j} \phi_{E,j}^{(2)} \Big)
+ \sum_{K\in\mathcal{T}^H} \sum_{j=m_K+1}^{M_K} d_{K,j} \phi_{K,j}^{(3)}
:= z_1 + z_2.
\end{equation*}
By the orthogonality of eigenfunctions of the spectral problem (\ref{eq:spectral}),
\begin{equation*}
\| z_1\|^2_V \lesssim
\sum_{E\in\mathcal{E}^H} \lambda_{E,n_E+1}^{-1} \int_E (v_h\cdot n)^2.
\end{equation*}

From (\ref{eq:rt2}), we have
\begin{equation}
\int_{\Omega} q \nabla \cdot v_h
= \int_{\Omega} f q - \int_{\Omega} \rho \frac{\partial p_h}{\partial t} q
\label{eq:bound1}
\end{equation}
for all $q\in Q_H^{(3)}$. Using the fact that $\int_K q = 0$ for all $q\in Q_H^{(3)}$, we have
\begin{equation*}
\int_{\Omega} q \nabla \cdot v_h
= \sum_{K\in\mathcal{T}^H} \sum_{j=1}^{M_K} d_{K,j} \int_K q \nabla \cdot \phi_{K,j}^{(3)}.
\end{equation*}
By the spectral problem (\ref{eq:local6}),
\begin{equation*}
\int_{\Omega} q \nabla \cdot v_h
= \sum_{K\in\mathcal{T}^H} \sum_{j=1}^{M_K} d_{K,j} \mu_{K,j} \int_K \rho \, p_{K,j}^{(3)} \, q.
\end{equation*}
Taking $q=p_{K,j}^{(3)}$, and using the condition that $\int_K \rho \, p_{K,i}^{(3)} \, p_{K,j}^{(3)} = \delta_{ij}$,
we see that
\begin{equation}
d_{K,j} \mu_{K,j} = \int_{\Omega} p_{K,j}^{(3)} \nabla \cdot v_h, \quad\quad j=1,2,\cdots,M_K, \quad K\in\mathcal{T}^H.
\label{eq:bound2}
\end{equation}
Thus, by the spectral problem (\ref{eq:local5})-(\ref{eq:local6}), we obtain
\begin{equation*}
\| z_2\|_V^2 \lesssim \sum_{K\in\mathcal{T}^H} \sum_{j=m_K+1}^{M_K} d_{K,j}^2 (\phi_{K,j}^{(3)},\phi_{K,j}^{(3)})_V
= \sum_{K\in\mathcal{T}^H} \sum_{j=m_K+1}^{M_K} d_{K,j}^2 \mu_{K,j}.
\end{equation*}
Using (\ref{eq:bound2}) and then (\ref{eq:bound1}),
\begin{equation*}
\| z_2 \|_V^2 \lesssim \sum_{K\in\mathcal{T}^H} \sum_{j=m_K+1}^{M_K} \int_{\Omega} d_{K,j} p_{K,j}^{(3)} \nabla \cdot v_h
= \int_{\Omega} \Big( f - \rho \frac{\partial p_h}{\partial t} \Big) \sum_{K\in\mathcal{T}^H} \sum_{j=m_K+1}^{M_K} d_{K,j} p_{K,j}^{(3)}.
\end{equation*}
Hence, we have
\begin{equation*}
\| z_2 \|_V^2 \lesssim \Big( \max_{K\in\mathcal{T}^H} \mu_{K,m_K+1}^{-1} \Big) \| \rho^{-1} f - \dot{p}_h \|_Q^2
\end{equation*}
since
\begin{equation*}
\| \sum_{K\in\mathcal{T}^H} \sum_{j=m_K+1}^{M_K} d_{K,j} p_{K,j}^{(3)} \|_Q^2
\leq \Big( \max_{K\in\mathcal{T}^H} \mu_{K,m_K+1}^{-1} \Big) \| z_2\|_V^2
\end{equation*}
by the spectral problem (\ref{eq:local5})-(\ref{eq:local6}).

This completes the proof of this theorem.
\end{proof}

\begin{theorem}
\label{thm:inter2}
The interpolants $\pi(p_h^{(I)})$ and $\pi(p_h^{(B)})$ defined in (\ref{eq:inter}) satisfy
\begin{eqnarray}
\| p_h^{(I)}-\pi(p^{(I)}_h) \|_Q^2 &\leq& \Big( \max_{K\in\mathcal{T}^H} \mu_{K,m_K+1}^{-1} \Big) \| \dot{v}_h \|_V^2, \label{eq:p1} \\
\| p_h^{(B)}-\pi(p^{(B)}_h) \|_Q^2 &\lesssim& h \Big(  \max_{E\in\mathcal{E}^0_p}{\lambda^{-1}_{E,n_E+1}}\Big) \| \dot{v}_h \|_V^2. \label{eq:p2}
\end{eqnarray}
\end{theorem}
\begin{proof}
By the definition of $\pi(p_h^{(I)})$, we have
\begin{equation*}
p_h^{(I)}-\pi(p^{(I)}_h) = \sum_{K\in\mathcal{T}^H} \Big( \sum_{j=m_K+1}^{M_K} a_{K,j} p_{K,j}^{(3)} \Big).
\end{equation*}
By the orthogonality of eigenfunctions $\{ p_{K,j}^{(3)} \}$ and the spectral problem
(\ref{eq:local6}), we have
\begin{equation*}
\| p_h^{(I)}-\pi(p^{(I)}_h) \|_Q^2 = (p_h^{(I)}-\pi(p^{(I)}_h), p_h^{(I)})_Q =
\int_{\Omega} p_h^{(I)} \nabla \cdot \Big( \sum_{K\in\mathcal{T}^H} \sum_{j=m_K+1}^{M_K} a_{K,j} \mu_{K,j}^{-1} \phi_{K,j}^{(3)} \Big).
\end{equation*}
We write $\displaystyle r_1 = \sum_{K\in\mathcal{T}^H} \sum_{j=m_K+1}^{M_K} a_{K,j} \mu_{K,j}^{-1} \phi_{K,j}^{(3)}$.
Using (\ref{eq:rt1}), we have
\begin{equation*}
\int_{\Omega} p_h^{(I)} \nabla \cdot r_1
= \int_{\Omega} \kappa \frac{\partial v_h}{\partial t} \cdot r_1
\leq \| \dot{v}_h \|_V \, \| r_1\|_V.
\end{equation*}
By the orthogonality of eigenfunctions $\{ \phi_{K,j}^{(3)} \}$,
\begin{equation*}
\|r_1\|_V^2 = \sum_{K\in\mathcal{T}^H} \sum_{j=m_K+1}^{M_K} (a_{K,j} \mu_{K,j}^{-1})^2 (\phi_{K,j}^{(3)},\phi_{K,j}^{(3)})_V
= \sum_{K\in\mathcal{T}^H} \sum_{j=m_K+1}^{M_K} (a_{K,j})^2 \mu_{K,j}^{-1} (p_{K,j}^{(3)},p_{K,j}^{(3)})_Q
\end{equation*}
which implies
\begin{equation*}
\|r_1\|_V^2 \leq \Big( \max_{K\in\mathcal{T}^H} \mu_{K,m_K+1}^{-1} \Big) \| p_h^{(I)}-\pi(p^{(I)}_h) \|_Q^2.
\end{equation*}
This proves (\ref{eq:p1}).

By the definition of $\pi(p_h^{(B)})$, we have
\begin{equation*}
p_h^{(B)}-\pi(p_h^{(B)}) = \sum_{E\in\mathcal{E}_p^0} \Big(  \sum_{j=n_E+1}^{N_E} b_{E,j} p_{E,j}^{(2)} \Big).
\end{equation*}
Recall that the basis functions $\{ p_{E,j}^{(2)} \}$ have supports in $\mathcal{K}^h_E$.
Thus, we have
\begin{equation*}
\| p_h^{(B)}-\pi(p_h^{(B)}) \|_Q^2 \lesssim h \| p_h^{(B)}-\pi(p_h^{(B)}) \|_{L^2(\mathcal{E}_p)}^2
\end{equation*}
where $\|\cdot\|_{L^2(\mathcal{E}_p)}$ denotes the $L^2$-norm defined on the set of coarse edges $\mathcal{E}_p$.
Next, we define $w\in \widehat{V}^{(2)}_{\text{snap}}$ such that
\begin{equation}
\label{eq:w}
[w\cdot n]|_{E}=-(p_h^{(B)}-\pi(p_h^{(B)}))|_{E}= \sum_{j=n_E+1}^{N_E} -b_{E,j} p_{E,j}^{(2)}
\end{equation}
for all $E \in \mathcal{E}^0_p$, where $\widehat{V}^{(2)}_{\text{snap}}$ is the union of all $\widehat{V}^{(2)}_{E, \text{snap}}$,
which is obtained by decoupling the continuity of normal components of basis functions in $V^{(2)}_{E, \text{snap}}$
on the edges in $\mathcal{E}_p^0$.
The condition (\ref{eq:w}) can be easily obtained by defining $w$ as
\begin{equation*}
w = \sum_{j=n_E+1}^{N_E} -b_{E,j} \phi_{E,j}^{(2)}
\end{equation*}
on one of the two coarse elements, denoted by $K_E$, having the edge $E$ and as zero on the other coarse element, for every $E \in \mathcal{E}_p^0$.
Using \eqref{eq:local4} and \eqref{eq:rt1}, we get
\begin{equation}
\label{eq:w1}
 \sum_{E\in\mathcal{E}_p^0} \int_E p^{(B)}_h\sum_{j=n_E+1}^{N_E} b_{E,j} p_{E,j}^{(2)} = \int_{\Omega} \kappa \frac{\partial v_h}{\partial t} \cdot w.
\end{equation}
In addition, by the orthogonality of eigenfunctions, it is easy to see that
\begin{equation}
\label{eq:w2}
 \| p_h^{(B)}-\pi(p_h^{(B)}) \|_{L^2(\mathcal{E}_p)}^2 = \sum_{E\in\mathcal{E}_p^0}  b_{E,j}^2 \sum_{j=n_E+1}^{N_E} \int_E (p_{E,j}^{(2)})^2 = \sum_{E\in\mathcal{E}_p^0} \int_E p^{(B)}_h\sum_{j=n_E+1}^{N_E} b_{E,j} p_{E,j}^{(2)}.
\end{equation}
On the other hand, by the Cauchy-Schwarz ineqaulity, we have
\[
 \int_{\Omega} \kappa \frac{\partial v_h}{\partial t} \cdot w \lesssim  \|\frac{\partial v_h}{\partial t}\|_V \, \| w\|_V
\]
and by the definition of $w$, we have
\begin{equation*}
\| w\|_V^2 \leq \sum_{E\in\mathcal{E}_p^0} \sum_{j=n_E+1}^{N_E} b_{E,j}^2 \int_{K_E} \kappa \phi_{E,j}^2
\leq \sum_{E\in\mathcal{E}_p^0} \sum_{j=n_E+1}^{N_E} b_{E,j}^2 \int_{\omega_E} \kappa \phi_{E,j}^2.
\end{equation*}
By the spectral problem \eqref{eq:spectral},
\begin{equation}
\label{eq:w3}
\int_{\omega_E} \kappa \phi_{E,j}^2 = \cfrac{1}{\lambda_{E,j}} \int_E (\phi_{E,j}\cdot n)^2 = \cfrac{1}{\lambda_{E,j}}\int_E (p_{E,j}^{(2)})^2.
\end{equation}
Combining \eqref{eq:w}, \eqref{eq:w1}, \eqref{eq:w2} and \eqref{eq:w3}, we obtain
\[
\| p_h^{(B)}-\pi(p_h^{(B)}) \|_{L^2(\mathcal{E}_p)}^2 \lesssim \Big( \max_{E\in\mathcal{E}^0_p}{\lambda^{-1}_{E,n_E+1}} \Big) \|\frac{\partial v_h}{\partial t}\|^2_V
\]
This completes the proof.
\end{proof}

Finally, we prove the following convergence theorem.
\begin{theorem}
Let $(v_H,p_H)$ be the solution of the mixed GMsFEM (\ref{eq:main3})-(\ref{eq:main4}),
and let $\pi(v_h)$, $\pi(p_h^{(I)})$ and $\pi(p_h^{(B)})$
be the interpolants defined in (\ref{eq:inter}).
We have
\begin{equation}
\begin{split}
&\: \max_{0\leq t\leq T} \Big( \| \pi(v_h)-v_H \|^2_V + \| \pi(p_h^{(I)}) + \pi(p_h^{(B)})-p_H\|^2_Q \Big) \\
\lesssim
&\:\Big( \max_{E\in\mathcal{T}^H} \lambda_{E,n_E+1}^{-1} \Big)  \sum_{E\in\mathcal{E}^H} \int_0^T \int_E (\dot{v}_h\cdot n)^2
 + h\Big( \max_{E\in\mathcal{T}^H} \lambda_{E,n_E+1}^{-1} \Big)  \sum_{E\in\mathcal{E}^H} \int_0^T \int_E (\ddot{v}_h\cdot n)^2 \\
 &\: +\Big( \max_{K\in\mathcal{T}^H} \mu_{K,m_K+1}^{-1} \Big) \int_0^T \Big( \| \rho^{-1} \dot{f} - \ddot{p}_h \|_Q^2  + \| \ddot{v}_h \|_V^2 \Big).
\end{split}
\end{equation}
\end{theorem}
\begin{proof}
The required bound is the consequence of (\ref{eq:errorbound}),
Theorem \ref{thm:inter1} and Theorem \ref{thm:inter2}.
\end{proof}

We remark that upper bounds for $\|v_h - v_H\|_V$ and $\| p_h - p_H \|_Q$
follow easily from the above theorem.

\section{Numerical Results}
\label{sec:num}

In this section, we present some numerical results
to show the performance of our method.
In our simulations, the computational domain $\Omega = [0,1]^2$, $\rho=1$ and
the source term $f$ is chosen as the Ricker wavelet
\[
f(x,t)= g(x) (t-2/f_{0})e^{-\pi^{2}f_{0}^{2}(t-2/f_{0}){}^{2}}, \quad g(x) = \delta^{-2} e^{((x-0.5)^{2}+(y-0.5)^{2})/\delta^{2}}
\]
where $f_0$ is the central frequency
and $\delta = 2h$ measures the size of the support of the source.
We assume that the initial conditions are zero.
In the first two examples, we will consider
the performance of our mixed GMsFEM
for the propagation of the above point source with two types of heterogeneities and
with the homogeneous Dirichlet boundary condition.
In our third example, we will apply
the perfectly matched layer (PML) \cite{berenger1994perfectly} to simulate
wave propagation in an unbounded domain containing a heterogeneous medium.

\subsection{A layered stochastic coefficient}
\label{heterogeneous case}
In our first example, we consider $\kappa^{-1}$ as a layered stochastic medium shown in Figure \ref{fig:m1}.
We will first construct the initial mesh $\mathcal{T}^H_0$ for the domain $\Omega=[0,1]\times[0,1]$.
To do so, we first define a
uniform triangular coarse mesh on $\Omega$ with mesh size $H=1/8$.
This triangular mesh is obtained by first dividing the domain into uniform squares
and then dividing each square into two triangles using the diagonal.
This process gives the initial mesh $\mathcal{T}^H_0$.
The required coarse mesh $\mathcal{T}^H$ is then obtained by the process described in Section \ref{sec:problem}.
Next, to construct the fine mesh $\mathcal{T}^h$,
each coarse triangular element in $\mathcal{T}^H$ is sub-divided into a union of uniform triangular fine mesh blocks with mesh size $h=1/64$
in the standard way.
Notice that, each coarse triangular element in $\mathcal{T}^H$ is divided into $64$ fine triangles.
We choose the source frequency $f_0=20$, and compute the solution at the time $T=0.2$.
The reference solution at time $T=0.2$ is shown in Figure \ref{fig:sol1}.
We apply our mixed GMsFEM (\ref{eq:main3})-(\ref{eq:main4}) for this problem with various choices of number of basis functions.
We will call the basis functions resulting from the first and the second basis sets the {\it boundary basis}
and the basis functions resulting from the third basis set the {\it interior basis}.
In Table \ref{tab:convergence}, we present the relative errors of the pressure $p$ in terms of the $Q$ norm
for various choices of number of basis functions.
In particular, we use $3$ to $6$ basis functions per coarse edge
for boundary basis
and use $4$ to $16$ basis functions per coarse element for interior basis.
We observe excellent performance of our method.
For reference, the fine grid solver has $24576$ unknowns for the pressure $p$
and $28928$ unknowns for the velocity $v$.
From Table \ref{tab:convergence}, we see that a very small dimensional approximation space
can give an error below $5\%$.
For instance, with the use of $4$ boundary basis functions per coarse edge
and the use of $12$ interior basis functions per coarse element,
the dimensions of the spaces $V_H$ and $Q_H$
are $7680$ and $5312$ respectively.
We remark that we observe a similar accuracy behavior for the velocity,
and we therefore skip those results in the paper.


\begin{figure}[ht]
\centering
\includegraphics[scale=0.5]{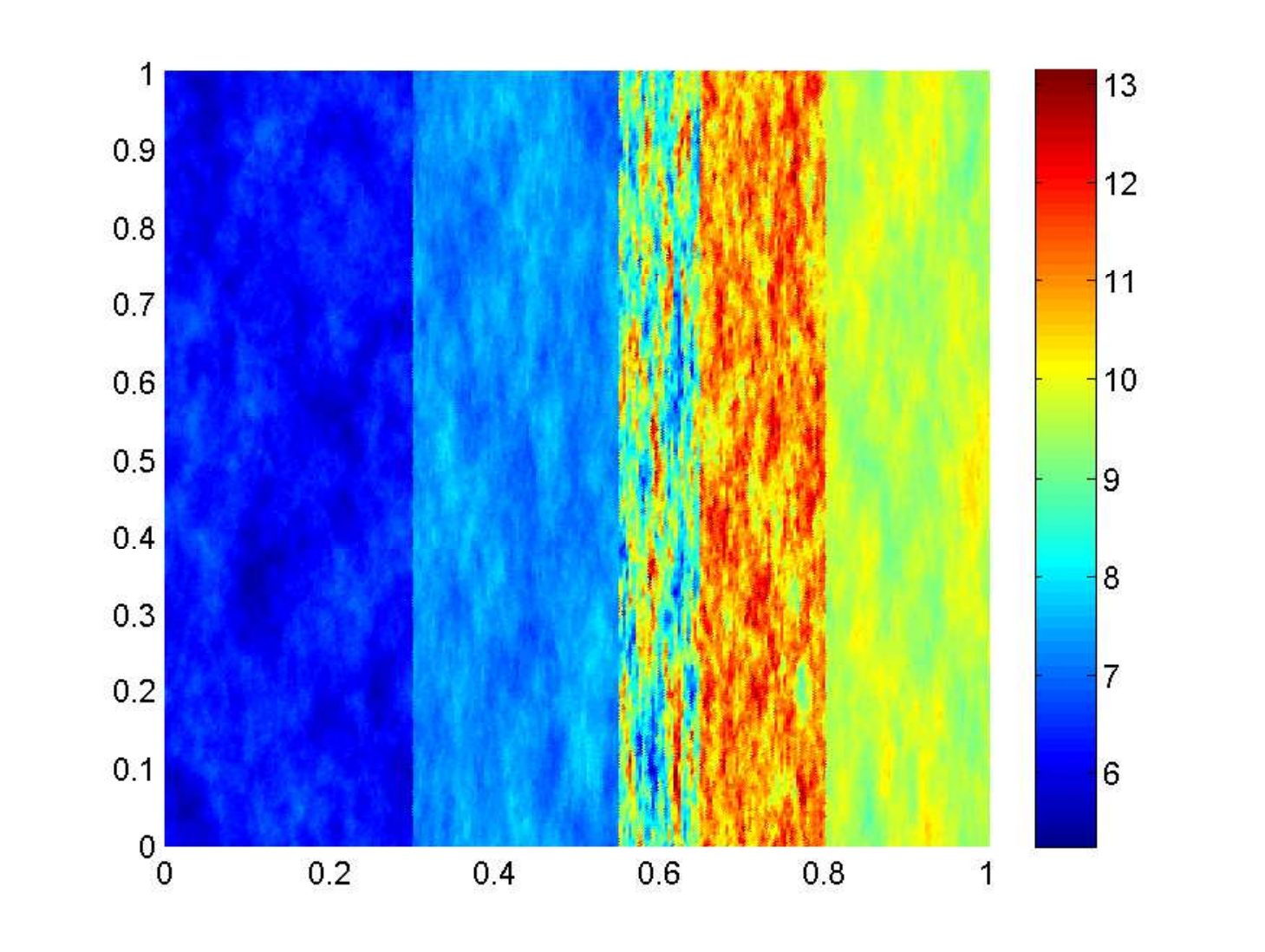}
\protect\caption{A stochastic coefficient for the first example.}
\label{fig:m1}
\end{figure}

\begin{figure}[ht]
\centering
\includegraphics[scale=0.5]{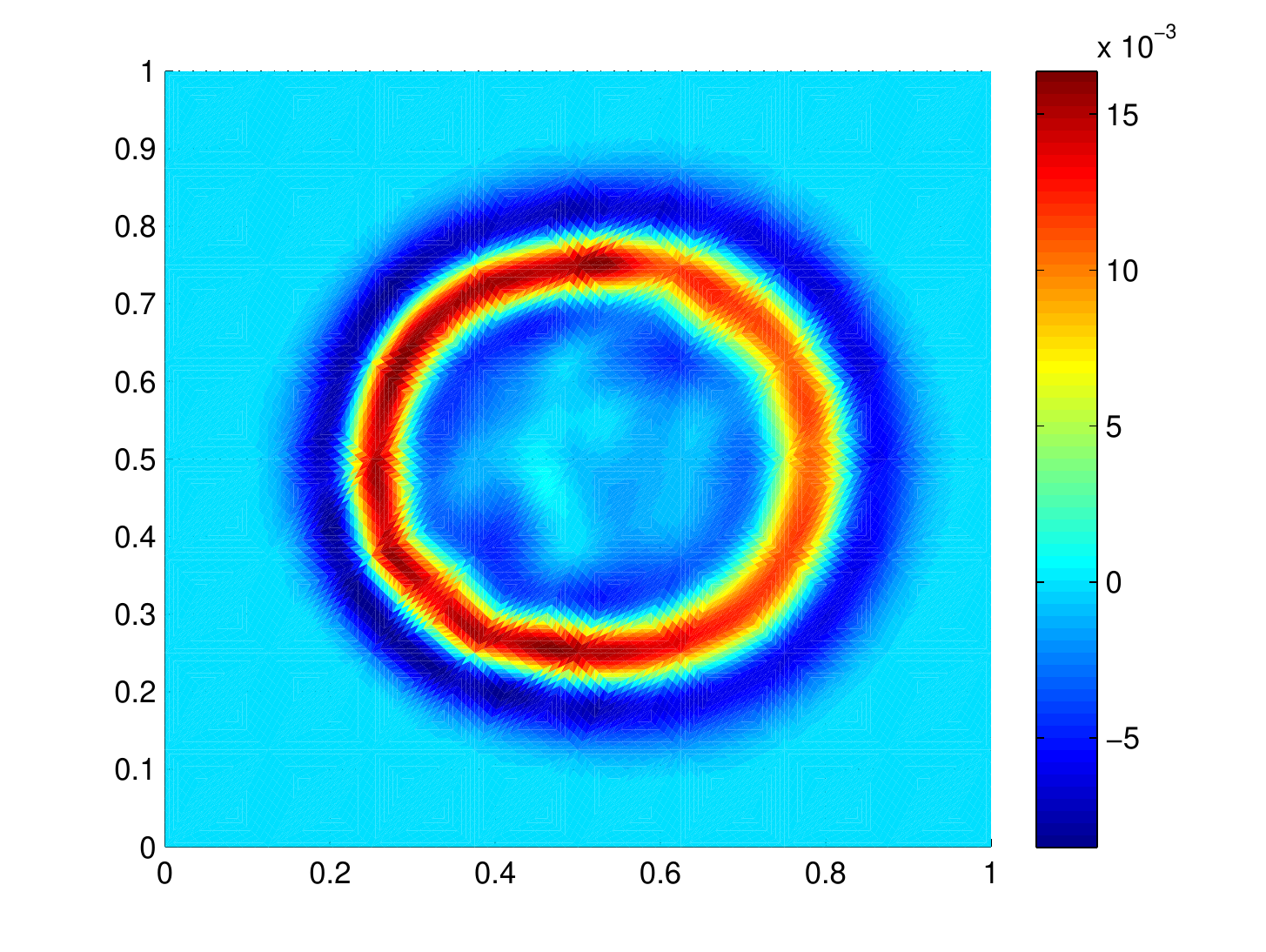}

\protect\caption{The reference solution for the first example.}
\label{fig:sol1}
\end{figure}

\begin{table}[ht]
\centering
\begin{tabular}{|c|c|c|c|c|}
\hline
\# boundary basis\textbackslash{}\# interior basis & 4 & 8 & 12 & 16\tabularnewline
\hline
3 & 13.22\% & 8.75\% & 7.84\% & 7.63\%\tabularnewline
\hline
4 & 12.76\% & 5.76\% & 3.65\% & 2.95\%\tabularnewline
\hline
5 & 12.97\% & 5.64\% & 3.31\% & 2.47\%\tabularnewline
\hline
6 & 13.01\% & 5.65\% & 3.31\% & 2.46\%\tabularnewline
\hline
\end{tabular}
\caption{Convergence history for various choices of number of basis functions for the first example.}
\label{tab:convergence}
\end{table}

\subsection{The Marmousi model}
\label{Marmousi case}

In our second example, we consider $\kappa^{-1}$ as a part of the Marmousi model shown in Figure \ref{fig:mar}. We assume that the domain $\Omega=[0,1]\times[0,1]$ is partitioned as the first example with coarse mesh size $H=1/16$ and the fine mesh size $h=1/256$.
We take the source frequency $f_0=20$, and compute the solution at the time $T=0.2$.
The reference solution and
the multiscale solution using
$6$ boundary basis functions per coarse edge and $12$ interior basis functions per coarse element
at the time $T=0.2$ are shown in Figure \ref{fig:sol2},
and we observe very good agreement.
In Table \ref{tab:convergence1}, we present the relative errors of the pressure $p$
in terms of the $Q$ norm
for various choices of number of basis functions.
In particular, we use $3$ to $6$ basis functions per coarse edge
for boundary basis
and use $4$ to $12$ basis functions per coarse element for interior basis.
We again observe excellent performance of our method.
For reference, the fine grid solver has $98304$ unknowns for the pressure $p$
and $135296$ unknowns for the velocity $v$.
From Table \ref{tab:convergence}, we see that with the use of $6$ boundary basis functions per coarse edge
and the use of $12$ interior basis functions per coarse element, the relative error for the pressure is $8.59\%$, and
the dimensions of the spaces $V_H$ and $Q_H$
are $36864$ and $22848$ respectively.

\begin{figure}[ht]
\centering
\includegraphics[scale=0.5]{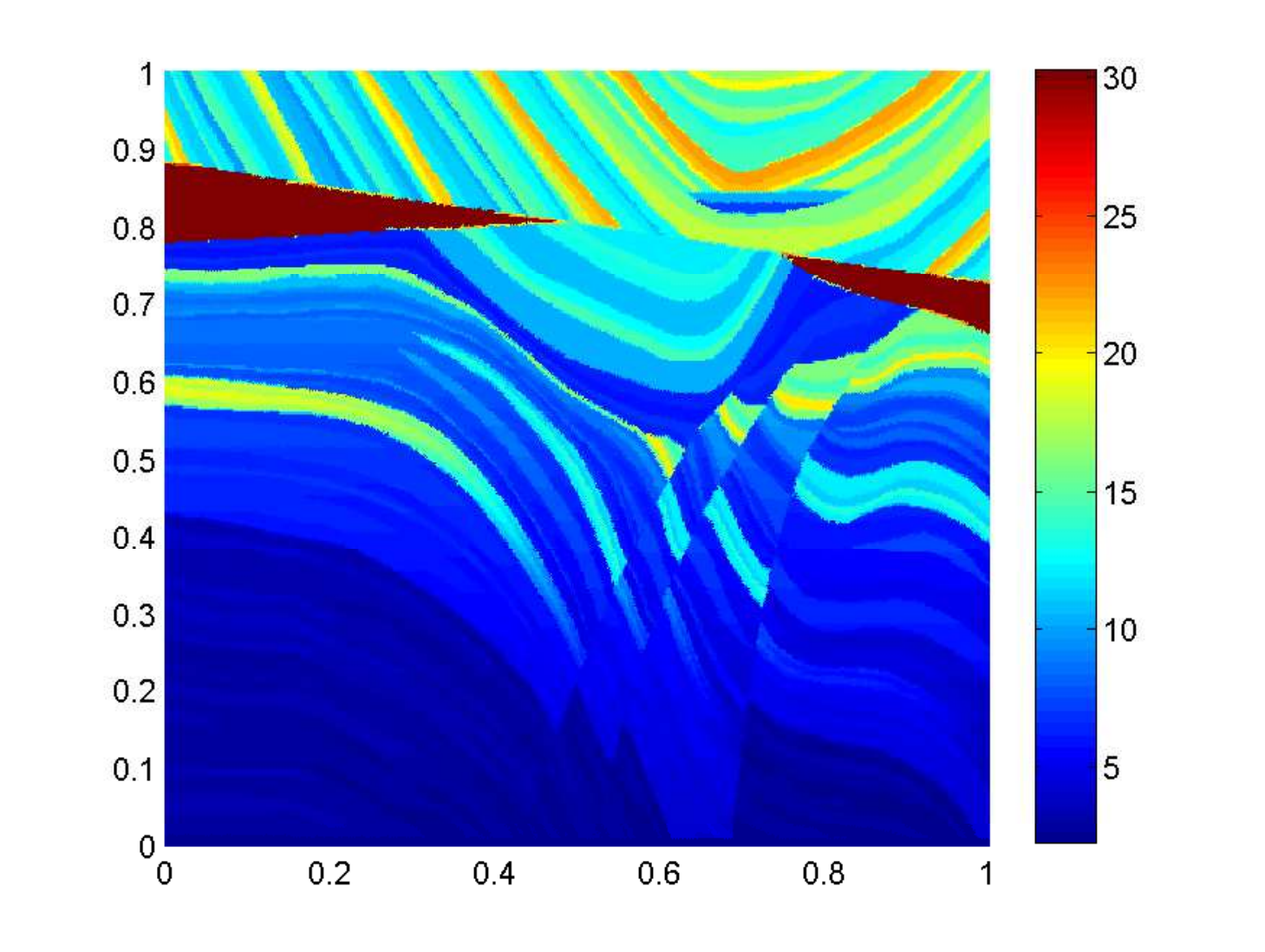}

\protect\caption{The Marmousi model for the second example.}
\label{fig:mar}
\end{figure}

\begin{figure}[ht]
\centering
\includegraphics[scale=0.5]{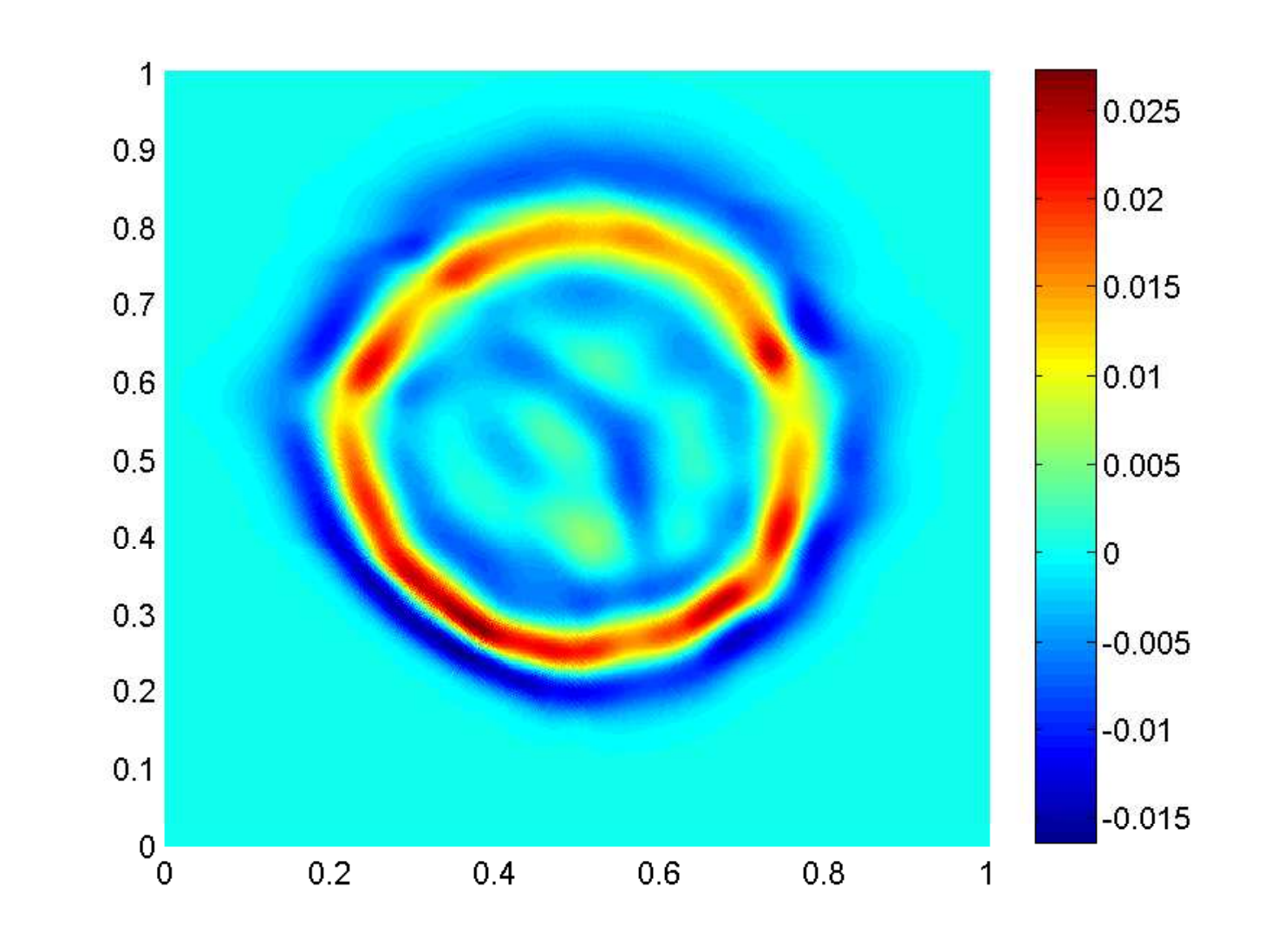} \includegraphics[scale=0.5]{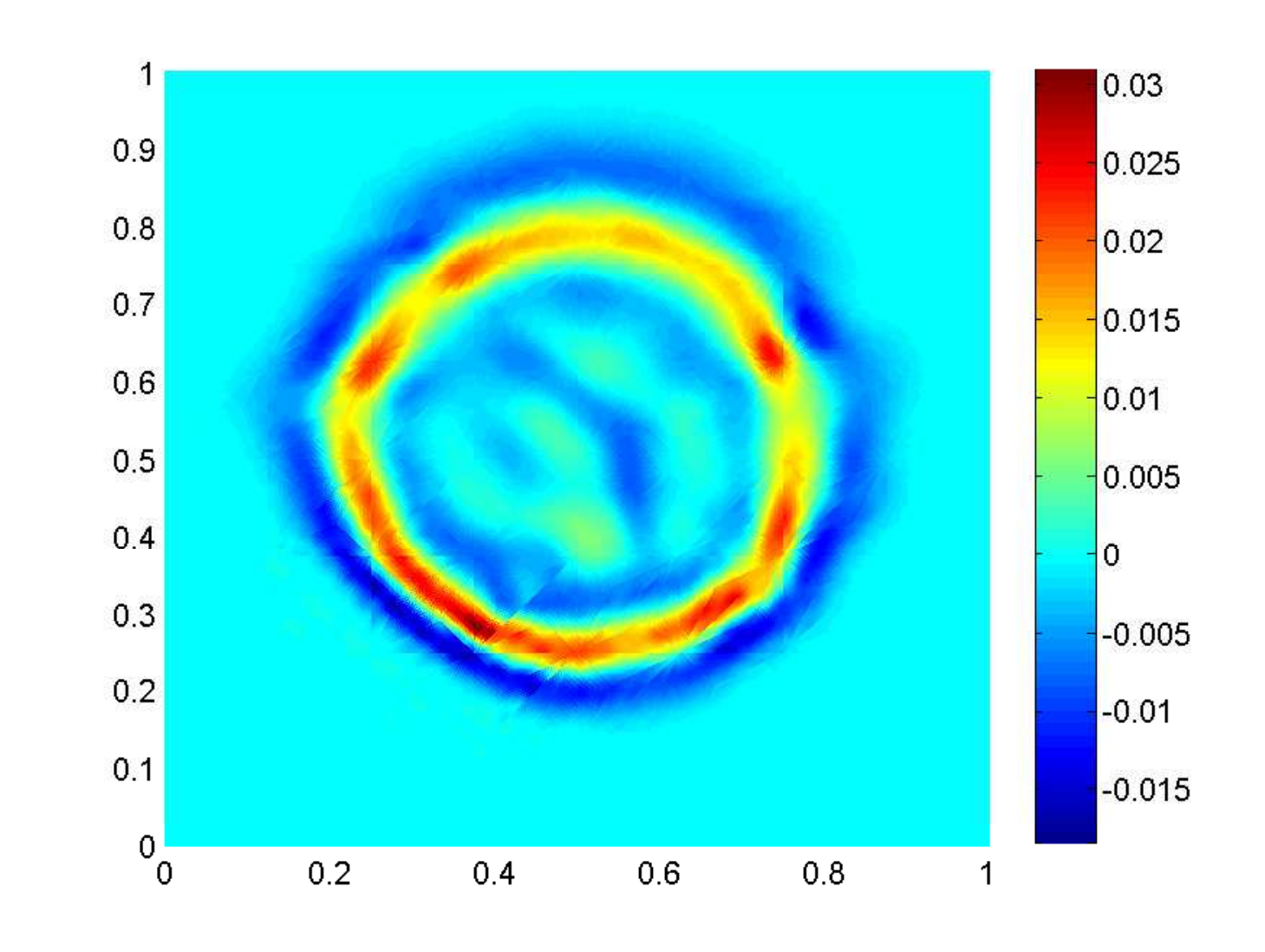}

\protect\caption{Left: The reference solution for the second example with $f_0=20$. Right: The corresponding multiscale solution for the second example using
$6$ boundary basis functions per coarse edge and $12$ interior basis functions per coarse element.}
\label{fig:sol2}
\end{figure}




\begin{table}[ht]
\centering
\begin{tabular}{|c|c|c|c|}
\hline
\# boundary basis\textbackslash{}\#interior basis & 4 & 8 & 12\tabularnewline
\hline
3 & 46.66\% & 45.64\% & 46.02\%\tabularnewline
\hline
4 & 38.11\% & 23.83\% & 23.39\%\tabularnewline
\hline
5 & 40.35\% & 13.66\% & 10.84\%\tabularnewline
\hline
6 & 41.58\% & 12.91\% & 8.59\%\tabularnewline
\hline
\end{tabular}
\caption{Convergence history for various choices of number of basis functions for the second example with $f_0=20$.}
\label{tab:convergence1}
\end{table}

To further test the performance of our method, we take a higher frequency source term
with $f_0=50$ and compute the solution at $T=0.16$.
The reference solution and the corresponding multiscale solution with $8$ boundary basis functions per coarse edge
and $20$ interior basis functions per coarse element
are shown in Figure \ref{fig:sol3}, and we observe very good agreement.
In addition, the relative errors of the pressure $p$ in terms of the $Q$ norm
for various choices of number of basis functions
are presented in Table \ref{tab:convergence2}.
In particular, we use $2$ to $8$ basis functions per coarse edge
for boundary basis
and use $4$ to $20$ basis functions per coarse element for interior basis.
From this table, we observe excellent performance of our method.

\begin{figure}[ht]
\centering
\includegraphics[scale=0.5]{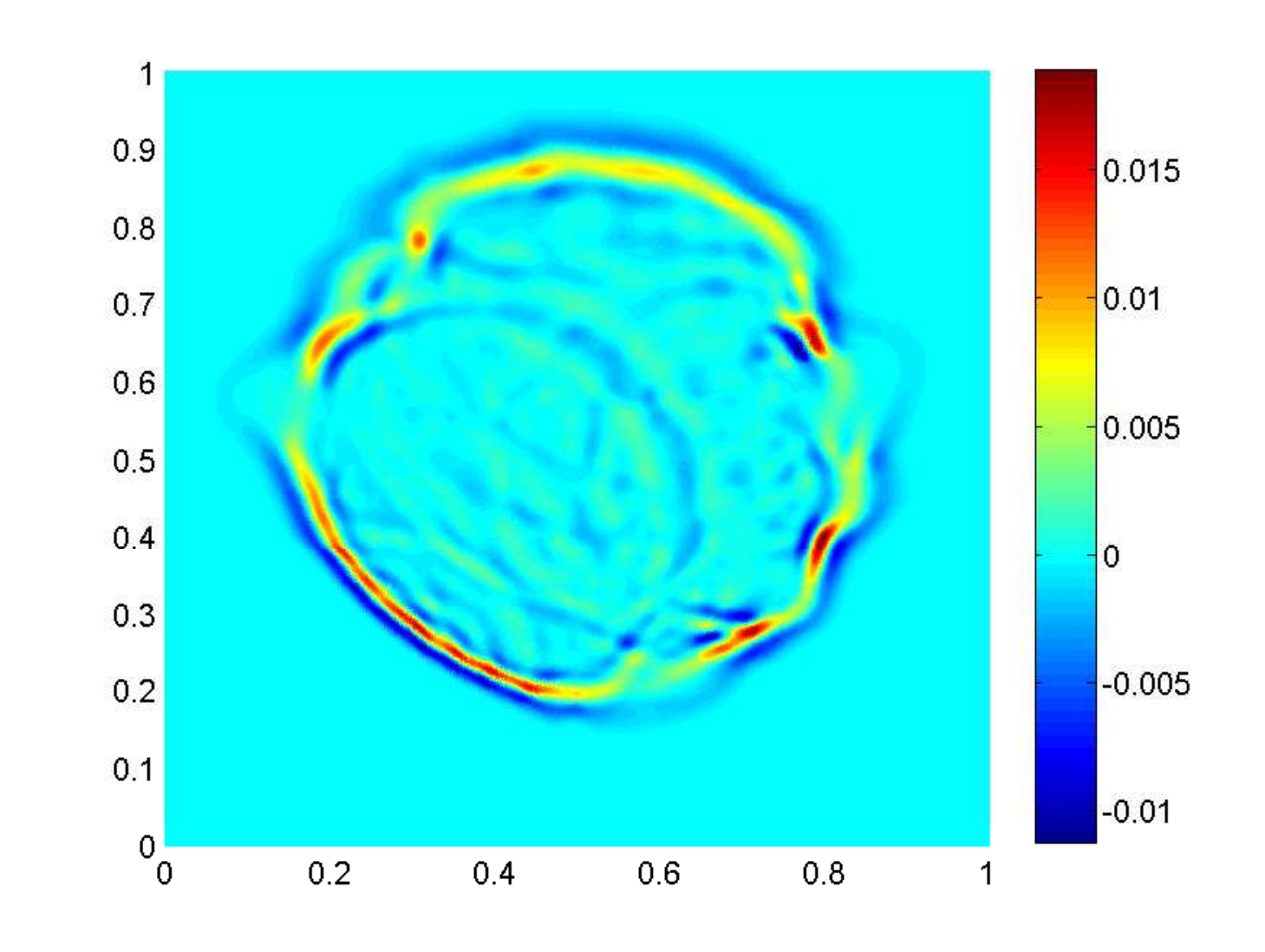} \includegraphics[scale=0.5]{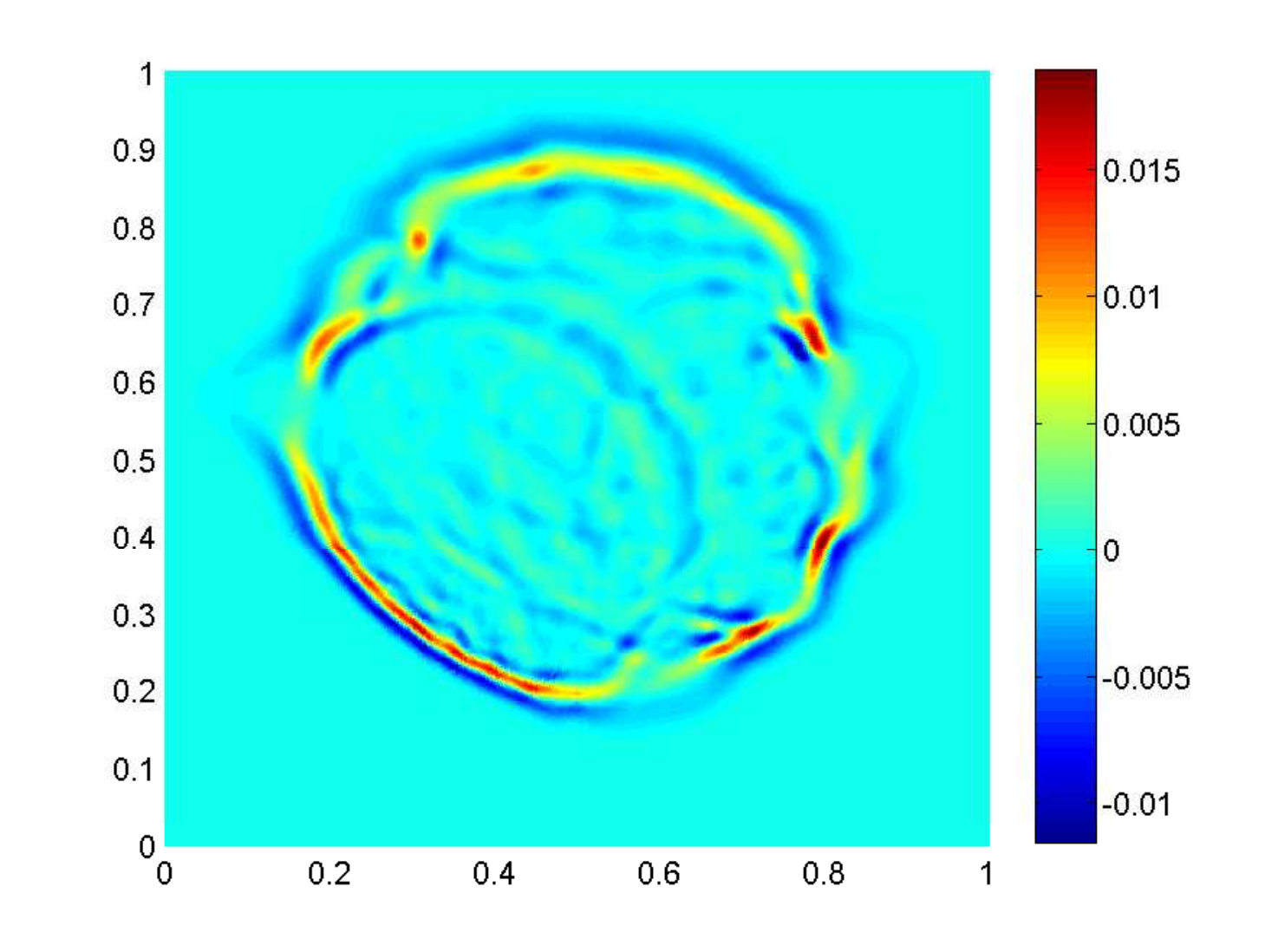}

\protect\caption{Left: The reference solution for the second example with $f_0=50$. Right: The corresponding multiscale solution for the second example using
$8$ boundary basis functions per coarse edge and $20$ interior basis functions per coarse element.}
\label{fig:sol3}
\end{figure}

\begin{table}[ht]
\centering
\begin{tabular}{|c|c|c|c|c|c|}
\hline
\# boundary basis\textbackslash{}\#interior basis & 4 & 8 & 12 & 16 & 20\tabularnewline
\hline
2 & 125.75\% & 128.19\% & 129.09\% & 129.29\% & 129.37\%\tabularnewline
\hline
4 & 91.55\% & 58.00\% & 61.29\% & 62.87\% & 63.48\%\tabularnewline
\hline
6 & 100.46\% & 36.11\% & 18.65\% & 16.95\% & 17.04\%\tabularnewline
\hline
8 & 101.04\% & 40.67\% & 16.95\% & 9.41\% & 6.92\%\tabularnewline
\hline
\end{tabular}
\caption{Convergence history for various choices of number of basis functions for the second example with $f_0=50$.}
\label{tab:convergence2}
\end{table}

\subsection{The Marmousi model with PML}
In the third example, we present an application of our method with the use of PML.
We assume that the computational domain $\Omega  = [0,1]^2$
and the medium is the part of the Marmousi model shown in Figure \ref{fig:mar}.
We take the source with $f_0=20$.
The coarse and fine mesh sizes for the computational domain $\Omega$ are $H=1/8$ and $h=1/64$ respectively.
To absorb the outgoing waves,
we use a PML with a width of $10$ fine grid blocks.
We apply our mixed GMsFEM within the computational domain $\Omega$
and the standard fine-grid method (\ref{eq:me1})-(\ref{eq:me2}) within the artificial layer.
In Figure \ref{fig:pml1}, we present the multiscale scale solutions
at different times $T$.
We see that the PML is able to absorb the outgoing waves without much artificial reflection.



\begin{figure}[ht]
\centering
\includegraphics[scale=0.5]{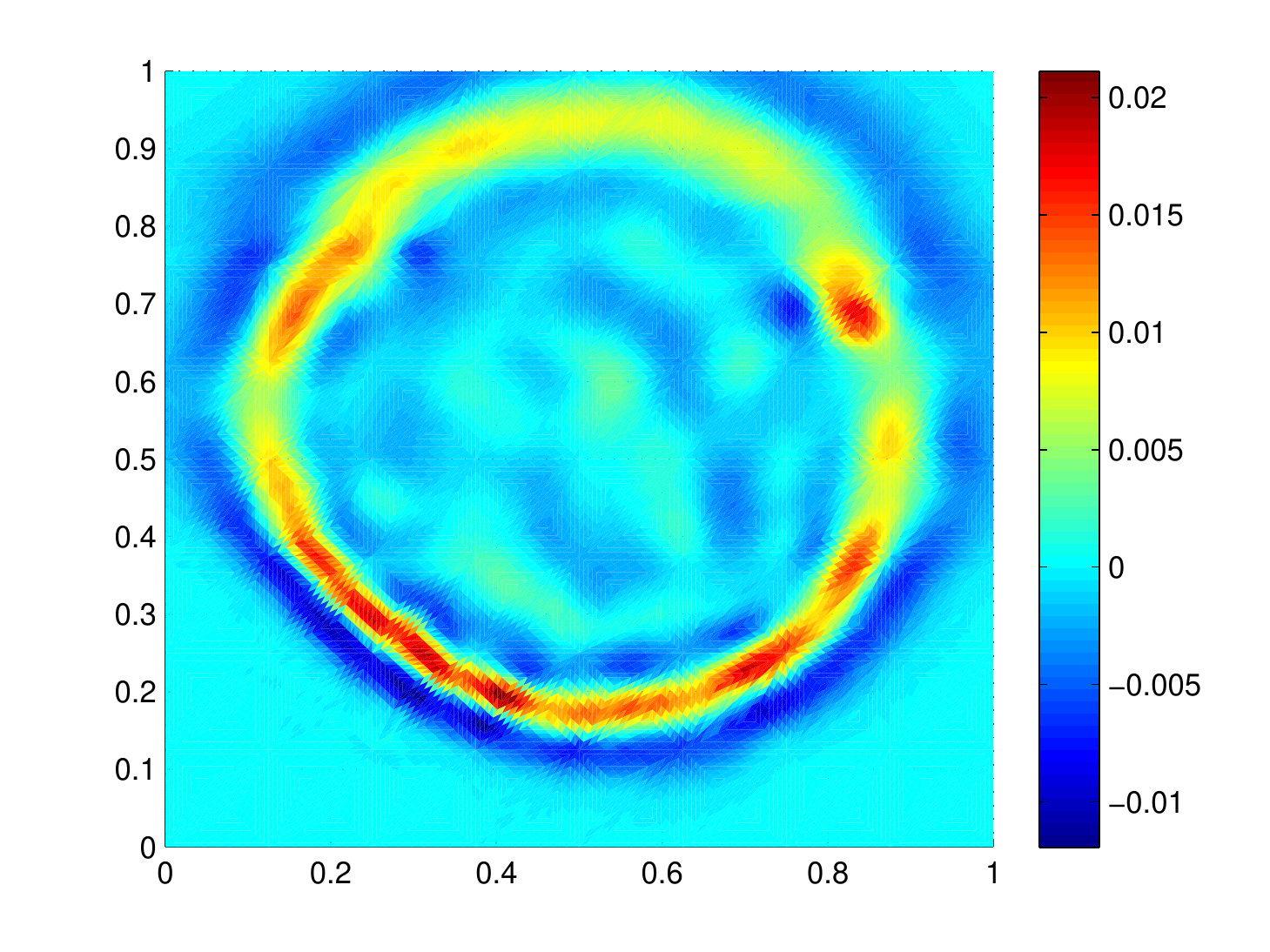} \includegraphics[scale=0.5]{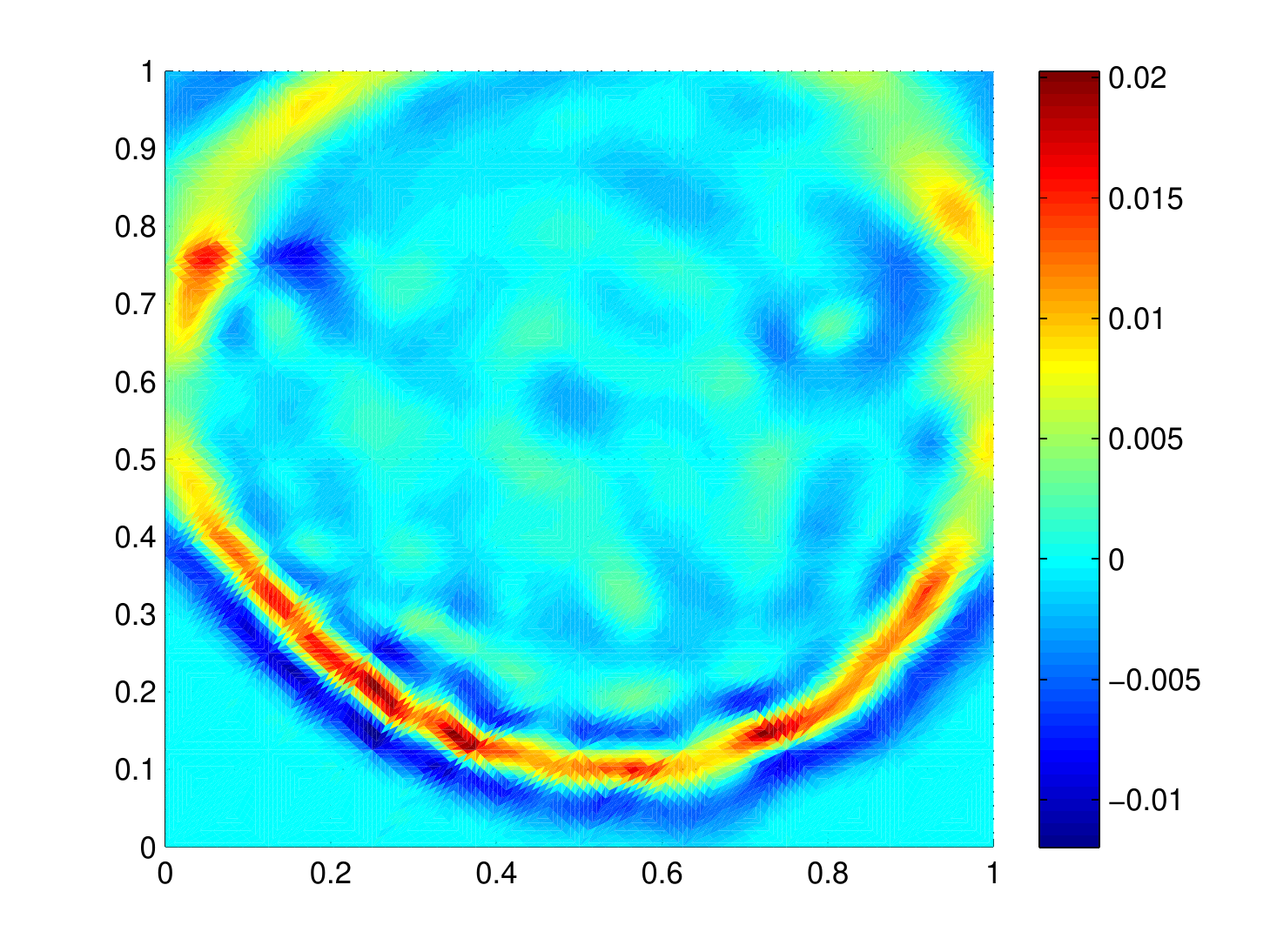}

\includegraphics[scale=0.5]{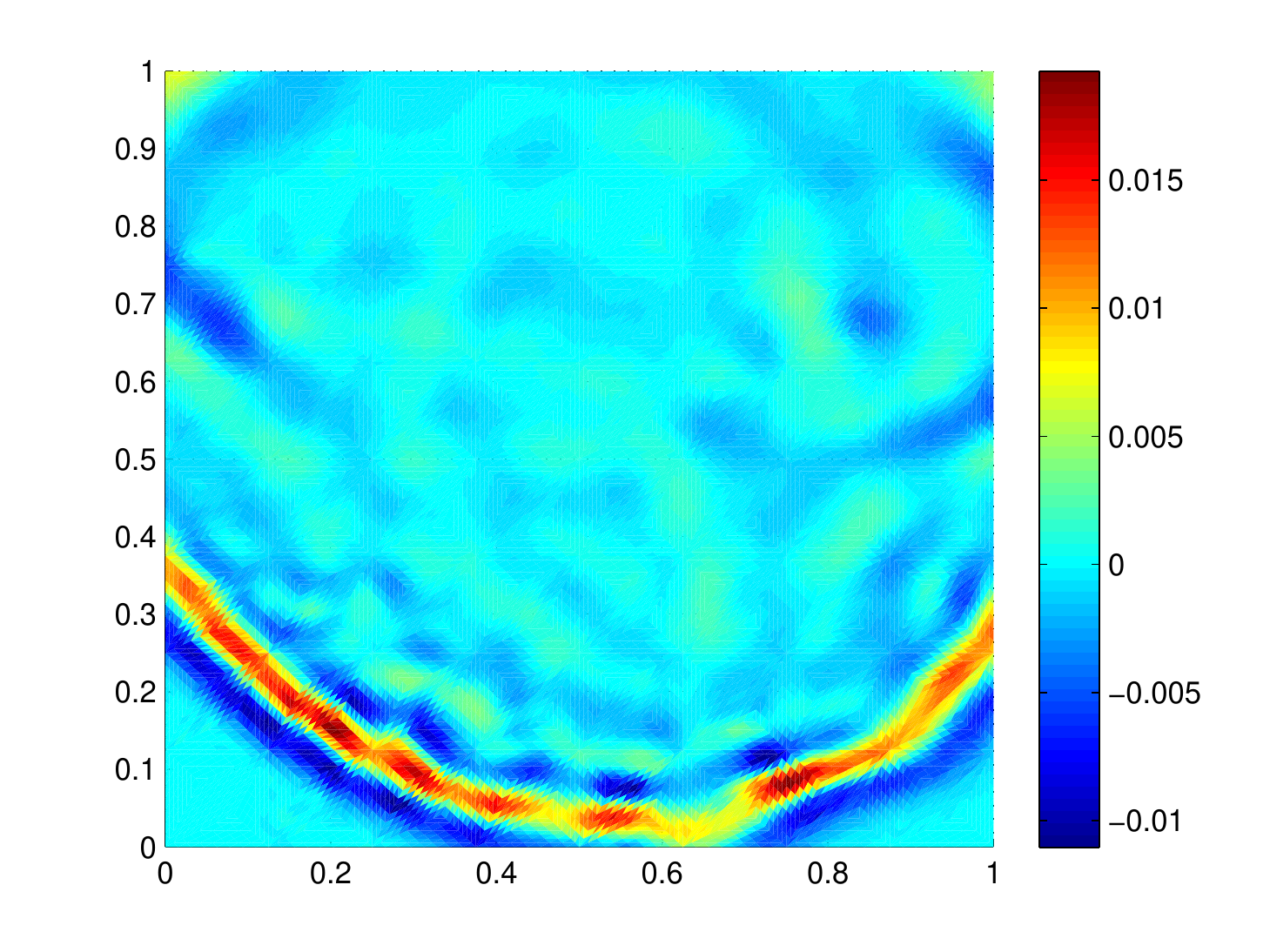}\includegraphics[scale=0.5]{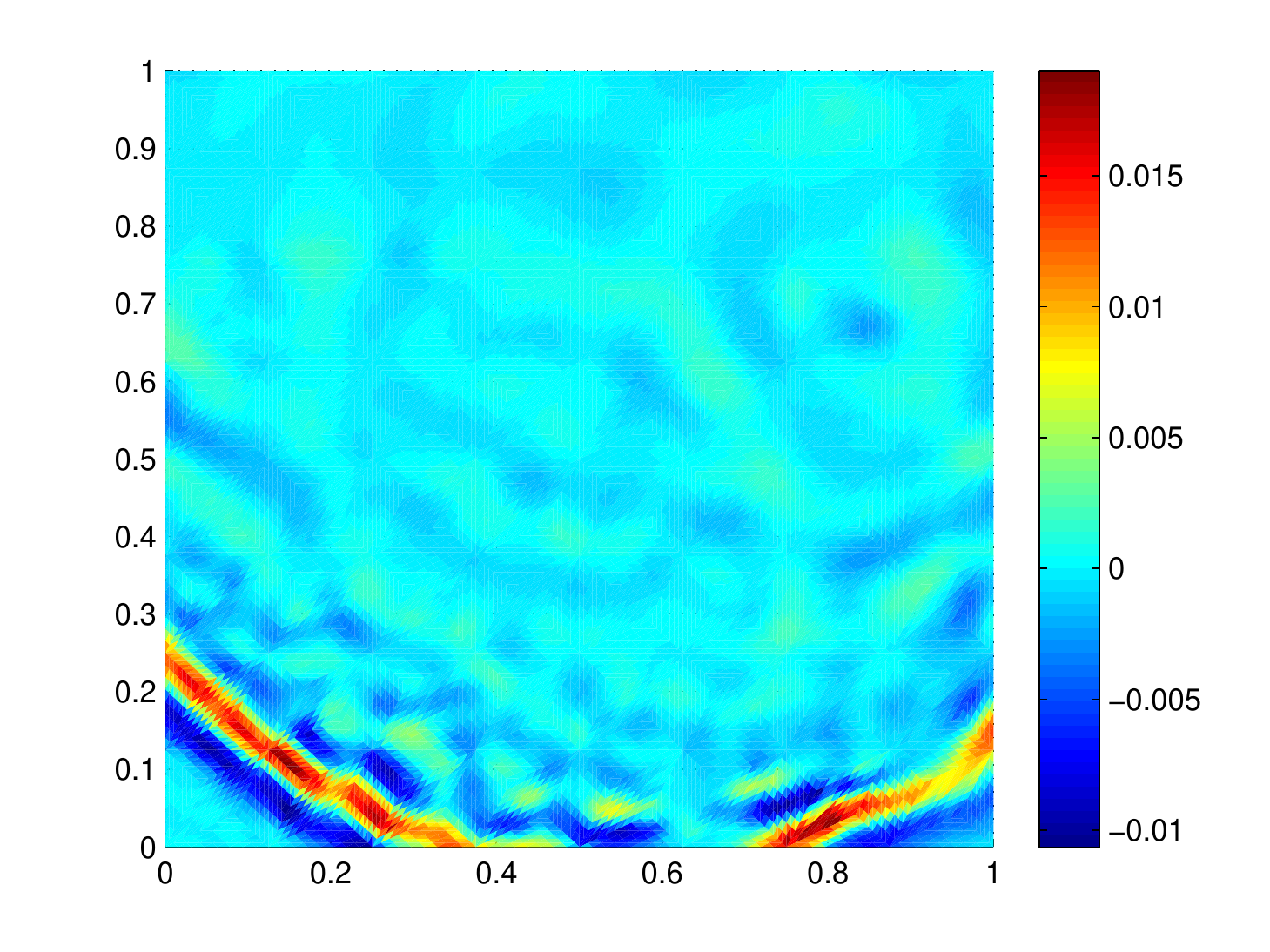}

\protect\caption{The multiscale solution at different times $T$ for the third example. Top-Left: $T=0.24$, Top-Right: $T=0.28$, Bottom-Left:
$T=0.32$, Bottom-Right: $T=0.36$.}
\label{fig:pml1}
\end{figure}




\section{Conclusion}

We develop and analyze a mixed GMsFEM for wave propagation in highly heterogeneous media.
The method is based on a mixed Galerkin global solver, and some local multiscale
basis functions for both the pressure and the velocity.
The multiscale basis functions are obtained by solving local spectral problems defined in some snapshot spaces.
The spectral problems give a natural ordering of the basis functions, which can be added
to the approximation space to give a spectral convergence.
By using a staggered coarse mesh and a carefully designed pair of multiscale basis spaces
with staggered continuity, our method is energy conserving and has block diagonal mass matrix.
Numerical results show excellent performance for our proposed approach.

\bibliographystyle{plain}
\bibliography{segabs_ceg,EricPaper,EricRef}

\end{document}